\documentclass[review]{elsarticle}

\usepackage[T1]{fontenc}
\usepackage[UKenglish]{babel}%
\usepackage{amssymb}
\usepackage{amsmath}
\usepackage{lmodern}
\usepackage{amsthm}
\usepackage{enumerate}
\usepackage{float}
\usepackage[usenames]{color}
\usepackage[all]{xy}
\usepackage{mathtools}
\usepackage{appendix}
\usepackage{pgf,tikz}
\usetikzlibrary{arrows}
\usepackage{multirow}
\usepackage{minitoc}
\usepackage{fullpage}
\usepackage{multicol}
\theoremstyle{definition}
\newtheorem{defi}{Definition} 
\newtheorem{rmq}[defi]{Remark}
\newtheorem{prop}[defi]{Proposition}
\newtheorem{lem}[defi]{Lemma}
\newtheorem{cor}[defi]{Corollary}                                   
\newtheorem{ex}[defi]{Example}

\newtheorem{thm}[defi]{Theorem}

\newcommand{\spec}{\mathfrak{G}}
\newcommand{\grad}{\nabla}
\newcommand{\vc}[1]{\mathbf{#1}}

\newcommand{\N}{\mathbb{N}}
\newcommand{\R}{\mathbb{R}}
\newcommand{\sR}{\mathfrak{R}}

\renewcommand{\S}{\mathbb{S}}

\newcommand{\Sum}{\displaystyle \sum}

\newcommand{\ps}[2]{\left\langle #1 , #2 \right\rangle}

\newcommand{\norm}[1]{\| #1 \|}
\newcommand{\Norm}[1]{\left\| #1\right\|}

\renewcommand{\bar}[1]{\overline{#1}}
\newcommand{\sign}{\textup{sign}}
\newcommand{\sgn}{\textup{sign}}

\renewcommand{\epsilon}{\varepsilon}

\newcommand{\NS}{\S}
\renewcommand{\and}{\quad \text{ and }\quad }
\newcommand{\mbd}{\gamma^-}
\newcommand{\Mbd}{\gamma^+}
\newcommand{\hap}{\circ}
\newcommand{\algo}{\text{HGPM}}
\renewcommand{\bigl}{\left(\rule{0cm}{4mm}\right.}
\renewcommand{\bigr}{\left)\rule{0cm}{4mm}\right.}

\newcommand{\msi}{[m]\setminus\{i\}}
\newcommand{\ba}{\boldsymbol{\alpha}}
\definecolor{darkOrange}{rgb}{0.9,0.2,0}
\definecolor{darkGreen}{rgb}{0.1,0.6,0.1}
\definecolor{gris}{rgb}{0.3,0.3,0.3}
\definecolor{commentgris}{rgb}{0.4,0.4,0.4}
\definecolor{lightgris}{rgb}{0.9,0.9,0.9}
\definecolor{ffqqqq}{rgb}{1,0,0}
\definecolor{fftttt}{rgb}{1,0.2,0.2}
\definecolor{Red}{rgb}{1,0,0}
\definecolor{Blue}{rgb}{0,0,1}
\definecolor{Green}{rgb}{0,1,0}
\definecolor{white}{rgb}{1,2,3}
\definecolor{Yellow}{rgb}{1,1,0}
\newif\iflongversion
\longversionfalse

\makeatletter
\def\ps@pprintTitle{%
 \let\@oddhead\@empty
 \let\@evenhead\@empty
 \def\@oddfoot{}%
 \let\@evenfoot\@oddfoot}
\makeatother
\begin{document}
\iflongversion
$ \frac{2}{3} $
\fi

\title{Tensor norm and maximal singular vectors of non-negative tensors -\\ a Perron-Frobenius theorem, a Collatz-Wielandt characterization and a generalized power method}

\author[rvt]{Antoine Gautier\corref{cor1}}
\ead{ag@cs.uni-saarland.de}
\author[rvt]{Matthias Hein}

\cortext[cor1]{Corresponding author.}
\address[rvt]{Department of Mathematics and Computer Science, Saarland University, Saarbr{\"u}cken, Germany}


\begin{abstract}
We study the $l^{p_1,\ldots,p_m}$ singular value problem for non-negative tensors. We prove a general Perron-Frobenius
theorem for weakly irreducible and irreducible nonnegative tensors and provide a Collatz-Wielandt characterization of the maximal singular value. Additionally, we propose a higher order power method for the computation of the maximal singular vectors and show that it has an asymptotic linear convergence rate.
\end{abstract}

\maketitle

\section{Introduction}
In recent years an increasing number of applications of the multilinear structure of tensors has been discovered in several disciplines, e.g. higher order statistics, signal processing, biomedical engineering, etc. \cite{Applsurvey,Ex_intro,Survey2}. In this paper we study the maximal singular value problem for nonnegative tensors which is induced by 
the variational characterization of the projective tensor norm.
Let $f\in \R^{d_1\times \ldots \times d_m}$ and $1 < p_1,\ldots,p_m< \infty$, we consider the $\ell^{p_1,\ldots,p_m}$ singular values of $f$ defined by Lim \cite{Lim} as the critical points of the function $Q \colon \R^{d_1}\times\ldots\times\R^{d_m} \to \R$ given by
\begin{equation*}
Q(\vc x_1,\ldots,\vc x_m)\coloneqq\frac{|f(\vc x_1, \ldots, \vc x_m)|}{\norm{\vc x_1}_{p_1}\cdot \ldots \cdot \norm{\vc x_m}_{p_m}},
\end{equation*}
where $\vc x_i \in \R^{d_i},\norm{\cdot}_p$ denotes the $p$-norm and
\begin{equation*}
f(\vc x_1, \ldots, \vc x_m)\coloneqq \Sum_{j_1 \in [d_1], \ldots, j_m \in [d_m]}f_{j_1,\ldots,j_m}x_{1,j_1}\cdot \ldots \cdot x_{m,j_m}
\end{equation*}
with $[n] \coloneqq \{1,\ldots,n\}$. The maximum of $Q$ is the so-called projective tensor norm \cite{defnorm} and we write it $\norm{f}_{p_1,\ldots,p_m}$. 
Note that the variational characterization of singular values we use here is slightly different than the one proposed in \cite{Lim} as we have the absolute value in $Q$ which
leads to the fact that singular values of tensors are all non-negative similar to the matrix case.

The main contributions of this paper are a  Perron-Frobenius Theorem for the maximal $\ell^{p_1,\ldots,p_m}$ singular value of nonnegative tensors together with its Collatz-Wielandt characterization and a power method that computes this maximal singular value and the associated singular vectors. More precisely, let $p'$ denote the H\"{o}lder conjugate of $p$ and for $n\in\N$, let $\psi_{p}:\R^{n}\to \R^{n}$ with $\big(\psi_{p}(\vc x)\big)_{j}=|x_{j}|^{p-1}\sign(x_{j})$ for $j\in [n]$. Moreover, let $\R^{n}_{++}\coloneqq \{\vc x \in \R^n\mid x_i>0, i\in [n]\},\NS^{d}_{++}\coloneqq \big\{(\vc x_1, \ldots, \vc x_m)\mid \vc x_k \in \R^{d_k},\norm{\vc x_k}_{p_k}=1\text{ and } \vc x_k\in \R^{d_k}_{++}, k \in [m]\big\}$ and for $i\in [m],k\in \msi, j_k\in [d_k]$ let $s_{i,k,j_k}\colon\NS^{d}_{++}\to\R$ be defined by
\begin{equation*}
s_{i,k,j_k}(\vc x)\coloneqq \psi_{p_k'}\left(\sum_{j_i\in[d_i]}\psi_{p_i'}\left(\frac{\partial}{\partial x_{i,j_j}}f(\vc x)\right)\frac{\partial^2}{\partial x_{i,j_j}\partial x_{k,j_k}}f(\vc x)\right),
\end{equation*}
then we have the following result.
\begin{thm}\label{gen_sum_PF}
Let $f\in\R^{d_1\times \ldots \times d_m}$ be a nonnegative weakly irreducible tensor and $1<p_1,\ldots,p_m< \infty$ such that there exists $i\in [m]$ with
\begin{equation*}
m-1 \leq (p_i-1)\left(\min_{k \in [m]\setminus\{i\}} p_k-(m-1)\right).
\end{equation*}
Then there exists a unique vector $\vc x^*\in \NS^d_{++}$ such that $Q(\vc x^*)=\norm{f}_{p_1,\ldots,p_m}$. Furthermore, it holds
\begin{equation*}
\max_{\vc x \in \NS^{d}_{++}}\ \prod_{\nu\in[m]\setminus\{i\}} \min_{j_\nu\in [d_\nu]} \left(\frac{s_{i,\nu,j_\nu}(\vc x)}{x_{\nu,j_\nu}}\right)^{p_\nu-1}= \norm{f}^{p_i'(m-1)} _{p_1,\ldots,p_m} =\min_{\vc x \in \NS^{d}_{++}}\ \prod_{\nu\in[m]\setminus\{i\}} \max_{j_\nu\in [d_\nu]} \left(\frac{s_{i,\nu,j_\nu}(\vc x)}{x_{\nu,j_\nu}}\right)^{p_\nu-1},
\end{equation*}
and if $f$ is irreducible then $\vc x^*>0$ is the unique nonnegative $\ell^{p_1,\ldots,p_m}$ singular vector of $f$, up to scale.
\end{thm}
Note that our theory can be extended so that the maximum (respectively the minimum) in the min-max characterization of Theorem \ref{gen_sum_PF} is taken over $\big\{(\vc x_1,\ldots,\vc x_m)\mid \vc x_k \in \R^{d_k}_{++} , k \in [m]\big\}$ instead of $\NS^d_{++}$. A key element of our proof is the construction of a bijection between the $\ell^{p_1,\ldots,p_m}$-singular vectors of $f$ and the points $\vc x \in \NS^{d-d_i}_{++}\coloneqq \big\{(\vc x_1, \ldots,\vc x_{i-1},\vc x_{i+1},\ldots, \vc x_m)\mid \vc x_k \in \R^{d_k}_{++}\text{ and }\norm{\vc x_k}_{p_k}=1, k \in \msi\big\}$ satisfying $s_{i,k,j_k}(\vc x) = \lambda^{p_i'(p_k'-1)}x_{k,j_k}$ for every $k\in\msi$ and $j_k\in [d_k]$ (see Proposition \ref{dual_egual}). Based on this observation we also build an algorithm that computes the maximal singular vectors of $f$, which is in the next theorem.
\begin{thm}\label{convrate_thm}
Let $f\in\R^{d_1\times \ldots \times d_m}$ be a nonnegative weakly irreducible tensor and $1<p_1,\ldots,p_m< \infty$ satisfying the assumption of Theorem \ref{gen_sum_PF}. Let $(\lambda_-^k)_{k\in\N},(\lambda_+^k)_{k\in\N}$ and $\big(\vc x^k\big)_{k\in\N}$ be the sequences produced by the Higher-order Generalized Power Method (see p.\pageref{HGPM_alg}). Moreover, let $\tilde{\vc x}^* \in \NS^{d}_{++}$ be a singular vector of $f$ satisfying $Q(\tilde{\vc x}^*)=\norm{f}_{p_1,\ldots,p_m}$ and let $\big(\norm{f}_{p_1,\ldots,p_m},\vc x^*\big)=\Phi_i^{-1}\big(\norm{f}_{p_1,\ldots,p_m},\tilde{\vc x}^*\big)$ where $\Phi_i$ is the bijection given in Proposition \ref{dual_egual}. Then 
\begin{equation*}
\forall k \in \N,\quad \lambda_-^k  \,\leq \lambda_-^{k+1}  \,\leq \, \norm{f}_{p_1,\ldots,p_m}  \, \leq\, \lambda_+^{k+1} \leq \, \lambda_+^k, \qquad
\lim_{k\to \infty} \lambda_-^k =\norm{f}_{p_1,\ldots,p_m}= \lim_{k\to \infty} \lambda_+^k , \qquad \lim_{k\to \infty} \vc x^k = \vc x^*
\end{equation*}
and there exists $0<\nu<1$, $k_0\in \N$ and a norm $\norm{\cdot}_G$ such that 
\begin{equation*}
\norm{\vc x^{k+1}-\vc x^*}_G \leq \nu \norm{\vc x^{k}-\vc x^*}_G \qquad \forall k \geq k_0.
\end{equation*}
\end{thm}
Note that the computation of matrix norms \cite{HenOls2010} and tensor norms \cite{Lim2013} is NP-hard in general and thus the restriction to nonnegative matrices resp. tensors is crucial.
In the case of matrices ($m=2$) a power method for the computation of a general $(p_1,p_2)$-norm of a nonnegative matrix and its associated singular vectors 
has been considered by Boyd \cite{Boyd} already in 1974. Recently, the paper has been reconsidered by \cite{Bhaskara} where uniqueness of the strictly positive singular vectors is
shown under the condition that the matrix is strictly positive. Our Perron-Frobenius theorem extends this uniqueness result to irreducible matrices (note that the notion of weakly
irreducible and irreducible coincide for matrices and reduce to the standard notion of irreducibility).

On the tensor side ($m>2$) Perron-Frobenius Theorems for nonnegative tensors have already been established for $H$-eigenvalues \cite{Chang} ($p_i=m, \, i=1,\ldots,m$) and general $\ell^{p_1,\ldots,p_m}$ singular vectors \cite{Fried} for $p_i\geq m,\, i=1,\ldots,m$. Our Perron-Frobenius theorem in Theorem \ref{gen_sum_PF} extends the range of $(p_1,\ldots,p_m)$ to
\begin{equation*}
m-1 \leq (p_i-1)\left(\min_{k \in [m]\setminus\{i\}} p_k-(m-1)\right).
\end{equation*}
In particular, this allows to have one $p_i$ to be arbitrarily close to $1$ given that the other $p_k$ are sufficiently large or alternatively, all except one $p_i$ can be 
arbitrarily close to $m-1$ whereas $p_i$ has to be sufficiently large. Moreover, our Collatz-Wielandt characterization seems to be the first one for the case where the $p_i$
are not all equal. Our result also leads to a slight generalization for the singular vectors of a partially symmetric nonnegative tensor introduced in \cite{Qi_rect_eig} which is discussed in more detail in Section \ref{Other_spectra_pb}.

The power method for tensors was first introduced by Ng, Qi and Zhou in \cite{Ng} for the computation of the maximal $H$-eigenvalue of an irreducible nonnegative tensor. It was generalized in \cite{Fried} where the method can also be used for computing $\ell^{p_1,\ldots,p_m}$ singular vectors for weakly primitive nonnegative tensors. However, their method applies only in the case when $p_1 =\ldots = p_m$ while our higher order power method needs only weak irreducibility of the nonnegative tensor and does not require $p_1,\ldots,p_m$ to be equal.

Let us describe the organization of this paper. In Section \ref{Notations}, we prove some general properties of the $\ell^{p_1,\ldots,p_m}$ singular values of $f$. In Section \ref{nnegtens}, we restrict our study to nonnegative tensors and provide criteria to guarantee the existence of some strictly positive singular vector\footnote{For simplicity we speak in the following of the singular vector even though this corresponds to a set of $m$ singular vectors of a $m$-th order tensor.} of $f$. In Section \ref{thethm}, we provide our Perron-Frobenius Theorem
and Collatz-Wielandt characterization of the maximal singular value. We discuss the relation between $\ell^{p_1,\ldots,p_m}$-singular value and other spectral problems for tensors in Section \ref{Other_spectra_pb}. The higher order power method together with its convergence rate are introduced in Section \ref{conv_sec}. Finally, in Section \ref{numexp} we do a small numerical experiment and compare our power method to the one proposed in \cite{Fried}.
\section{Notations and characterization of the singular spectrum}\label{Notations}
For $f\in\R^{d_1,\ldots,d_m}$, $f \geq 0$ (resp. $f>0$) mean that every entry of $f$ is nonnegative (resp. strictly positive), furthermore  we write $f \leq g$ (resp. $f < g$) if $g-f \geq 0$ (resp. $g-f>0$). 
Let $\sR^d \coloneqq\R^{d_1}\times \ldots \times \R^{d_m}$, we use bold letters without index to denote vectors in $\sR^d$, bold letters with index $i\in [m]$ denote vectors in $\R^{d_i}$ and the components of these vectors are written in normal font, i.e. $\vc x \in \sR^d, \vc x = (\vc x_1,\ldots,\vc x_m), \vc x_i \in \R^{d_i}, \vc x_i =(x_{i,1},\ldots,x_{i,d_i}), x_{i,j_i} \in \R$. 
We denote by $\NS^d$ the ``unit sphere'' in $\sR^{d}$, i.e. $\vc x \in \NS^d$ if and only if $\norm{\vc x_i}_{p_i}=1$ for every $i \in [m]$. We write $p'$ to denote the H\"{o}lder conjugate of $1<p<\infty$ (i.e. $\frac{1}{p}+\frac{1}{p'}= 1$). For $i \in [m]$, set 
$\sR^{d-d_i}\coloneqq\R^{d_1}\times \ldots \times \R^{d_{i-1}}\times \R^{d_{i+1}}\times \ldots \times \R^{d_m}$
and let $\NS^{d-d_i}$ be the set of $\vc x\in \sR^{d-d_i}$ such that $\norm{\vc x_k}_{p_k}=1$ for every $k \in [m]\setminus\{i\}$. Furthermore, for $n\in\N$ and $V\in\big\{\sR^d,\NS^{d},  \sR^{d-d_i},\NS^{d-d_i},\R^{n}\big\}$ we write $V_+$ (resp. $V_{++}$) the restriction of $V$ to the positive cone (resp. to the interior of the positive cone), e.g. $\sR^{d}_{+} \coloneqq \{\vc x \in \sR^d\mid \vc x \geq 0\}$ and $\NS^{d-d_i}_{++} \coloneqq \{\vc x \in\NS^{d-d_i}\mid \vc x > 0 \}$.
We follow a similar system of notation as vectors for the gradient of $f$, that is $\grad f(\vc x)\in \sR^d, \grad_i f(\vc x)\in \R^{d_i}, \partial_{i,j_i} f(\vc x)\in \R , \grad f(\vc x) = (\grad_1f(x),\ldots,\grad_mf(\vc x))$ and $\grad_if(\vc x) = \big(\partial_{i,1}f(\vc x), \ldots, \partial_{i,d_i}f(\vc x)\big)$ where $\partial_{i,j_i}\coloneqq \frac{\partial}{\partial x_{i,j_i}}$. In particular, note that $f(\vc x)=\ps{\grad_if(\vc x)}{\vc x_i}$ and $\partial_{i,j_i}f(\vc x) = f(\vc x_1,\ldots,\vc x_{i-1},\vc e_{(i,j_i)},\vc x_{i+1},\ldots, \vc x_{m})$, where $\vc e_{(i,1)},\ldots,\vc e_{(i,d_i)}$ is the canonical basis of $\R^{d_i}$.
Furthermore, for $g \in \R^{d_1\times \ldots \times d_m}$ we denote by $|g|\in \R^{d_1\times \ldots \times d_m}$ the tensor such that $|g|_{j_1,\ldots,j_m} = |g_{j_1,\ldots,j_m}|$ for every $j_1\in [d_1],\ldots,j_m \in [d_m]$. 

For $1<q < \infty$ define $\psi_q:\R^n \to \R^n$ by $(\psi_q(\vc y))_j = |y_j|^{q-1}\sign(y_j)$ for every $1 \leq j \leq n$. Note that $\psi_q\big(\psi_{q'}(\vc x)\big) = \vc x,\norm{\psi_{q}(\vc x)}_{r} = \norm{\vc x}_{r(q-1)}^{q-1}$ and $\grad \norm{\vc y}_q=\norm{\vc y}_q^{1-q}\psi_q(\vc y)$.  Consider the function $S: \sR^d \to \R$ defined by $S(\vc x) =\norm{\vc x_1}_{p_1}\cdot \ldots \cdot \norm{\vc x_m}_{p_m}$. 

One can conclude that a critical point $\vc x$ of $Q$ such that $f(\vc x) \neq 0$ ($Q$ is not differentiable at $\vc x$ if $f(\vc x)=0$) must satisfy the following nonlinear system of equations
\begin{equation}\label{eigensys}
 \sgn\big(f(\vc x)\big)\grad_{i}f(\vc x)=Q(\vc x)S(\vc x)\norm{\vc x_i}^{-p_i}_{p_i}\psi_{p_i}(\vc x_i) \qquad \forall i \in [m].
\end{equation}
Now, for $i \in [m]$ consider
\begin{equation}\label{def_sig}
\sigma_i \colon \sR^d \to \R^{d_i} ,\quad \vc x \mapsto \sgn\big(f(\vc x)\big)\psi_{p_i'}\big(\grad_{i}f(\vc x)\big),
\end{equation}
then $\vc x\in\NS^d$ is a critical point of $Q$ if and only if $f(\vc x) \neq 0$ and $\sigma_i(\vc x) = Q(\vc x)^{p_i'-1}\vc x_{i}$ for $i\in[m]$. This motivates our choice for the definition of $\spec(f)$, the set of critical values of the function $Q$, as follow: 
\begin{equation*}
\spec(f)\coloneqq \left\{ \lambda \in \R\setminus\{0\} \mid \exists\vc x \in \NS^d \text{ with } \sigma_i(\vc x) = \lambda^{p_i'-1}\vc x_{i} \ \forall i \in [m]\right\}.
\end{equation*}
The proof of Proposition \ref{boundedspectrum} resembles the study of $Z$-eigenvalues in \cite{Quynhn}.
\begin{prop}\label{boundedspectrum}
Let $1<p_1,\ldots,p_m< \infty$ and $f\in \R^{d_1\times\ldots \times d_m}$. If $f \neq 0$, then $\spec(f)\neq \emptyset$ and for every $\lambda \in \spec(f)$ we have $0< \lambda \leq \min_{i \in [m]}\max_{l_i \in [d_i]} d_i^{1/p_i'}\partial_{i,l_i} \tilde f(\mathbf{e}),$
where $\vc e = (1,1,\ldots,1) \in \sR^d$ and $\tilde f\coloneqq |f|$.
\end{prop}
\begin{proof}
First of all, note that for every $\vc x \in \sR^d$ such that $S(\vc x)>0$ we have $Q(\vc x)=Q\left(\frac{\vc x_1}{\norm{\vc x}_{p_1}},\ldots,\frac{\vc x_m}{\norm{\vc x}_{p_m}}\right)$. Thus, $\norm{f}_{p_1,\ldots,p_m}=\sup_{\vc x\in \NS^d} Q(\vc x)$ and since $\vc x \mapsto Q(\vc x)$ is a continuous function on the compact set $\NS^d$, it reaches its maximum. It follows that $\spec(f) \neq \emptyset$. Let $\vc x^* \in \NS^d$ be a critical point of $Q$ associated to $\lambda \in \spec(f)$. Let $i \in [m]$, then, from $\norm{\vc x^*_i}_{p_i}=1$, it follows that there exists some $j_i \in [d_i]$ such that $|x^*_{i,j_i}| \geq d_i^{-1/p_i}$. In particular, the fact that $\psi_{p_i}$ is an increasing function implies $\big|\psi_{p_i}(x^*_{i,j_i})\big| =\psi_{p_i}\big(|x^*_{i,j_i}|\big)\geq \psi_{p_i}\big(d_i^{-1/p_i}\big)$. Since $\vc x^* \in \NS^d$ is a critical point of $Q$, we have $\big|\lambda \ \psi_{p_i}(x^*_{i,j_i})\big|=\big|\partial_{i,j_i} f(\vc x^*)\big|$. Thus,
\begin{equation*}
\lambda = \lambda \ d_i^{1/p_i'-1/p_i'} =\lambda \ d_i^{1/p_i'} \psi_{p_i}\big(d_i^{-1/p_i}\big)\leq d_i^{1/p_i'}\big|\lambda \ \psi_{p_i}(x_{i,j_i}^*)\big|=d_i^{1/p_i'}\big|\partial_{i,j_i} f(\vc x^*)\big| \leq d_i^{1/p_i'} \max_{l_i \in [d_i]} \big|\partial_{i,l_i} f(\vc x^*)\big|.
\end{equation*}
Since $\norm{\vc x^*_k}_{p_k}=1$, we must have $|x^*_{k,j_k}|\leq 1$ for every $k\in [m]$ and $j_k\in[d_k]$, i.e. $|\vc x^*| \leq \vc e$, where $\vc e \in \sR^d$ is the vector who's entries are all $1$. Hence, for every $l_i \in [d_i]$, it holds $\big|\partial_{i,l_i} f(\vc x^*)\big| \leq \partial_{i,l_i} \tilde f\big(|\vc x^*|\big)\leq \partial_{i,l_i}\tilde f(\mathbf{e})$. Taking the minimum over all $i \in [m]$ concludes the proof.
\end{proof}
Note that for every $f\neq 0$ and every $1<p_1,\ldots,p_m<\infty$, there exists always at least one singular value.
\begin{ex}
Let $A \in \R^{2 \times 2}$ be the matrix defined by $A_{1,1}=A_{2,2}=0$ and $A_{1,2}=-A_{2,1} = 1$. For every $1 < p_1,p_2 < \infty$, the vector $
\vc x^* = \left(2^{-\frac{1}{p_1}},-2^{-\frac{1}{p_1}},2^{-\frac{1}{p_2}},2^{-\frac{1}{p_2}}\right) \in \R^2 \times \R^2$ is a $\ell^{p_1,p_2}$ singular vector of $A$ associated to the singular value $\lambda = 2^{1-\frac{1}{p_1}-\frac{1}{p_2}}$.
\end{ex}
We formulate now an equivalent characterization of $\spec(f)$. Note that for every $\vc z_i,\vc y_i\in\R^{d_i}$ we have $
\grad_if(\vc x_1,\ldots, \vc x_{i-1},\vc y_i,\vc x_{i+1},\ldots, \vc x_m)=\grad_if(\vc x_1,\ldots, \vc x_{i-1},\vc z_i,\vc x_{i+1},\ldots, \vc x_m),$ thus we may, without ambiguities, abuse notation and write $\grad_if(\vc x)$ regardless if $\vc x \in \sR^d$ or $\vc x\in \sR^{d-d_i}$. For $i\in[m]$ and $\vc x \in\sR^{d-d_i}$, let $S_i$ and $Q_i$ be the functions defined by $S_i(\vc x)\coloneqq\norm{\vc x_1}_{p_1} \cdot \ldots \cdot \norm{\vc x_{i-1}}_{p_{i-1}}\norm{\vc x_{i+1}}_{p_{i+1}} \cdot \ldots \cdot \norm{\vc x_m}_{p_m}$
 and
\begin{equation*}
Q_i\colon\big\{\vc x\in \sR^{d-d_i}\mid S_i(\vc x)> 0\big\}\to\R,\quad \vc x \mapsto \frac{\norm{\grad_if(\vc x)}_{p_i'}}{S_i(\vc x)}.
\end{equation*}
Observe that $\vc x \in \sR^{d-d_i}$ such that $S_i(\vc x) \neq 0$ and $\grad_if(\vc x)\neq 0$ is a critical point of $Q_i$ if and only if it satisfies
\begin{equation}\label{dual_eigensys}
S_i(\vc x)^{p_i'}\norm{\vc x_k}^{p_k}_{p_k}\ps{\psi_{p_i'}\big(\grad_{i}f(\vc x)\big)}{\grad_{i}\partial_{k,j_k}f(\vc x)} = Q_i(\vc x)^{p_i'}\psi_{p_k}(x_{k,j_k}), \qquad \forall k \in [m] \setminus \{i\}, j_k\in  [d_k].
\end{equation}
Here, we used $\partial_{i,j_i}\partial_{k,j_k}f(\vc x) = \partial_{k,j_k}\partial_{i,j_i}f(\vc x)$ as $f \in C^2(\sR^d)$. It follows that every $\vc x \in \NS^{d-d_i}$ such that $\grad_if(\vc x)\neq 0$ is a critical point of $Q_i$ if and only if $s_{i,k}(\vc x) = Q_{i}(\vc x)^{p_i'(p_k'-1)} \vc x_k$ for every $k \in [m]\setminus \{i\}$, where
\begin{equation}\label{def_s}
s_{i,k}\colon\sR^{d-d_i}\to\R^{d_k},\quad \vc x\mapsto\psi_{p'_k}\bigg(\grad_{k}f \bigl\vc x_1,\ldots, \vc x_{i-1},\psi_{p_i'}\big(\grad_{i}f(\vc x)\big),\vc x_{i+1},\ldots, \vc x_m\bigr \bigg).
\end{equation}
In the next proposition we show equivalence between the critical points of $Q_i$ and those of $Q$.
\begin{prop}\label{dual_egual}
Let $1<p_1,\ldots,p_m< \infty$ and $f\in \R^{d_1 \times \ldots \times d_m}$ with $f \neq 0$, denote by $C^*\subset \spec(f) \times \NS^d$ the set of pairs $(\lambda, \vc x^*)$ where $\vc x^*$ is a critical point of $Q$ associated to the critical value $\lambda $. Furthermore, for $i \in [m]$, let $\spec_i(f) \coloneqq\big\{ \lambda \in \R\setminus\{0\}\mid\exists\vc x \in \NS^{d-d_i} \text{ with } s_{i,k}(\vc x)= \lambda^{p_i'(p'_k-1)} \vc x_k \ \forall k\in [m]\setminus\{i\}\big\}
$ and $C^*_i \subset \spec_i(f) \times \NS^{d-d_i}$ be the set of all pairs $(\lambda, \vc x^*)$ where $\vc x^*$ is a critical point of $Q_i$ associated to the critical value $\lambda$.
Then $\spec(f)=\spec_i(f)$ and, with $
\varsigma_i(\vc x) \coloneqq \sign\bigg(f\left(\vc x_1,\ldots,\vc x_{i-1},\psi_{p_i'}\big(\lambda^{-1}\ \grad_{i}f(\vc x)\big),\vc x_{i+1},\ldots,\vc x_m\right)\bigg)$, the function
\begin{equation*}
\Phi_i\colon C_i^* \to C^*,\quad (\lambda,\vc x) \mapsto \bigg(\lambda, \left(\vc x,\ldots,\vc x_{i-1},\psi_{p_i'}\big(\varsigma_i(\vc x)\ \lambda^{-1}\ \grad_{i}f(\vc x)\big),\vc x_{i+1},\ldots,\vc x_m\right)\bigg)
\end{equation*}
is a bijection.
\end{prop}
\begin{proof} 
First, we show that $\Phi_i$ is well defined. Let $(\lambda,\vc x) \in C_i^*$ and $(\lambda, \vc x^*) \coloneqq \Phi_i(\lambda,\vc x)$. Using Equation \eqref{dual_eigensys}, $\grad_if(\vc x) = \grad_if(\vc x^*), \vc x_i^* =\psi_{p_i'}\big(\varsigma_i(\vc x)\ \lambda^{-1}\ \grad_{i}f(\vc x)\big)$ and $\grad_i\partial_{k,j_k}f(\vc x) = \grad_i\partial_{k,j_k}f(\vc x^*)$, one can show that $\lambda \psi_{p_k}(\vc x^*_k) = \varsigma_i(\vc x)\grad_kf(\vc x^*)$ for every $k\in[m], j_k \in [d_k]$. Furthermore, $\sgn\big(f(\vc x^*)\big) =\sgn\big(\ps{\grad_{i}f(\vc x^*)}{\vc x_i^*}\big)=\varsigma_i(\vc x) \sign\left(\lambda^{1-p'_i}\norm{\grad_{i}f(\vc x)}_{p_i'}^{p_i'}\right) = \varsigma_i(\vc x)$ and $\norm{\vc x_i^*}_{p_i}=1$ because $\norm{\grad_if(\vc x)}_{p_i'}^{p_i'-1}=Q_i(\vc x)^{p_i'-1}=\lambda^{p_i'-1}$, so $(\lambda,\vc x^*)\in C^*$. Injectivity is straightforward by definition of $\Phi_i$ and surjectivity is shown by noticing that if $(\lambda,\vc x^*) \in C^*$ then $(\lambda,\vc x^*)=\Phi_i\big(\lambda, (\vc x_1^*,\ldots,\vc x_{i-1}^*,\vc x_{i+1}^*,\ldots, \vc x^*_m )\big) $. Finally, the fact that $\Phi_i$ is a bijection implies $\spec(f) = \spec_i(f)$.
\end{proof}
The previous proposition implies that the maximum value of $Q_i$ is $\norm{f}_{p_1,\ldots,p_m}$. Moreover, if we know a maximizer of $Q_i$ then we can construct a maximizer of $Q$ and vice-versa. This result can be seen as a generalization of the variational characterization of the singular values and singular vectors of matrix (see e.g. Theorem 8.3-1 \cite{Golub}).
\section{Nonnegative tensors and positive singular vectors}\label{nnegtens}
Now, we focus our study on tensors $f\in \R^{d_1,\ldots,d_m}$ with nonnegative coefficients, i.e. $f\geq 0$. 
\begin{lem}\label{order_preserving}
Let $f\in \R^{d_1 \times \ldots \times d_m}$ and $\vc x, \vc y\in\sR^{d}_{+}$ such that $f \geq 0$ and $0 \leq \vc x \leq \vc y$, then $0\leq f(\vc x)\leq f(\vc y)$, $0\leq \grad_k f(\vc x) \leq \grad_k f(\vc y)$, $0\leq\sigma_k(\vc x)\leq \sigma_{k}(\vc y)$ and $0 \leq s_{i,k}(\vc x)\leq s_{i,k}(\vc y)$ for every $k,i\in [m]$ with $k \neq i$. 
\end{lem}
\begin{proof}
Straightforward computation.
\end{proof}
We recall the definition of irreducible and weakly irreducible tensors.
\begin{defi}[\cite{Fried}]\label{def_irr}
Let $f \geq 0$ and  $G(f)=\big(V, E(f)\big)$ an undirected $m$-partite graph with $V\coloneqq\big(\{1\}\times [d_1]\big)\cup\ldots\cup\big(\{m\}\times [d_m]\big)$ and such that for every $k,l\in[m]$ with $k\neq l$, we have $\big((k,j_k),(l,j_l)\big)\in E(f)$ if and only if there exist $j_\nu \in [d_\nu ]$ for $\nu \in[m]\setminus\{ k,l\}$ with $f_{j_1, \ldots, j_m}> 0$.
\begin{enumerate}[i)]
\item We say that $f$ is irreducible if for each proper nonempty subset $\emptyset \neq J \subsetneq V$ the following holds: Let $I = V\setminus J$, then there exist $(k, j_k)\in J$ and $(l,j_l) \in I$ for $l \in [m]\setminus\{k\}$ such that $f_{j_1,\ldots,j_m} > 0$.
\item We say that $f$ is weakly irreducible if $G(f)$ is connected.
\end{enumerate}
\end{defi}
The next proposition lists some useful properties of nonnegative tensors. 
\begin{prop}\label{Characterizer_prop}
For $1< p_1,\ldots,p_m<\infty, f\in \R^{d_1 \times \ldots \times d_m}$ and 
$\ba\coloneqq (\alpha_0,\alpha_1,\ldots,\alpha_m)>0$, consider the function $T_{\ba}\colon\sR^d \to \sR^d$ defined 
by $T_{\ba}(\vc z)\coloneqq \alpha_0\vc z +\big(\alpha_1\sigma_1(\vc z), \ldots,\alpha_m\sigma_m(\vc z)\big).$
 \begin{enumerate}[a)]
\item $f \geq 0 $ if and only if $\alpha_0\vc z \leq T_{\ba}(\vc z)$ for every $\vc z \in \NS^d_+$. \label{fgeq0}
\item If $f \geq 0$ is weakly irreducible, then $\alpha_0\vc z < T_{\ba}(\vc z)$ for every $\vc z \in \NS^d_{++}$. \label{magiclema}
\item $f\geq 0$ is irreducible if and only if there is some $n \in \N$ such that $T_{\ba}^n(\vc z)> 0$ for every $\vc z \in \NS^d_+$, where $T_{\ba}^n(\vc x) $ is recursively defined by $T_{\ba}^n(\vc x) = T_{\ba}^{n-1}(T_{\ba}(\vc x))$. \label{f_irr}
\item $f > 0$ if and only if $T_{\ba}(\vc z)> 0$ for every $\vc z \in \NS^d_+$. \label{f>0}
\end{enumerate}
\end{prop} 
\begin{proof}
Let us recall that $\vc z \in \NS^d$ implies $\vc z_k \neq 0$ for every $k \in [m]$.
\begin{enumerate}[a)]
\item If $f \geq 0$ and $\vc z\in\NS^{d}_+$, then, by Lemma \ref{order_preserving}, $\sigma_i(\vc z) \geq 0$ for any $i\in [m]$ and thus $\alpha_0\vc z \leq T_{\ba}(\vc z)$. Now, suppose that there is $\vc z \in\NS_+^d,i \in [m]$ and $j_i \in [d_i]$ such that $\big(T_{\ba}(\vc z)\big)_{i,j_i} < \alpha_0z_{i,j_i}$. It follows that $
0 >\big(T_{\ba}(\vc z)\big)_{i,j_i}-\alpha_0z_{i,j_i} = \alpha_i\sigma_{i,j_i}(\vc z)+ \alpha_0z_{i,j_i} - \alpha_0z_{i,j_i} = \alpha_i\sigma_{i,j_i}(\vc z).$ Since $\alpha_i > 0$, this is possible if and only if $\partial_{i,j_i}f(\vc z) <0$. However, $\vc z \geq 0$, so there must be some entry of $f$ which is strictly negative.
\item Let $f\geq 0$ be weakly irreducible and $\vc z \in \NS^d_{++}$. We need to show that $\sigma_{k,j_k}(\vc z)>0$ for every $k \in [m]$ and $j_k \in [d_k]$. Suppose by contradiction that there is some $i\in [m]$ and $j_i \in [d_i]$ such that $\partial_{i,j_i} f(\vc z)=0$. Then, from $\vc z> 0$ and $f \geq 0$ follows that $f_{j_1,\ldots,j_m} = 0$ for every $j_k \in [d_k]$ and $k \in [m]\setminus\{i\}$, thus the vertex $(i,j_i)$ is not connected to any other vertex in the graph $G(f)$ associated to $f$. This is a contradiction to the fact that $f$ is weakly irreducible.
\item Suppose that $f\geq 0$ is irreducible and for $\vc z \in \NS^d_{+}$ let $J_{\vc z_i} \coloneqq\{j_i \in [d_i]\mid z_{i,j_i} = 0\}$ and $J_{\vc z}\coloneqq J_{\vc z_1}\times \ldots \times J_{\vc z_m}$. Note that $\vc z \in \NS^d$ implies $J_{\vc z_i} \neq [d_i]$ for every $i \in [m]$. $J_{T_{\ba}(\vc z)}\subset J_{\vc z}$ follows from $f\geq 0$ and the property discussed above. We show that if $J_{\vc z} \neq \emptyset$, then $J_{T_{\ba}(\vc z)}$ is strictly contained in $J_{\vc z}$. So, suppose $J_{\vc z} \neq \emptyset$ and assume by contradiction that $J_{T_{\ba}(\vc z)}=J_{\vc z}$. Let $\nu\in [m]$ and $j_{\nu}\in J_{\vc z_\nu}$, then $0 = \alpha_0z_{\nu,j_{\nu}}+\alpha_\nu \ \sigma_{\nu,j_\nu}(\vc z)=\alpha_{\nu} \  \psi_{p_\nu}\big(\partial_{z_{\nu,j_{\nu}}}f(\vc z)\big)$ and 
\begin{equation*}
\Sum_{\substack{j_1\in[d_1]\setminus J_1,\ldots,j_{\nu-1}\in [d_{\nu-1}]\setminus J_{\nu-1},\\ j_{\nu+1}\in [d_{\nu+1}]\setminus J_{\nu+1},\ldots,j_m\in [d_m]\setminus J_m}}  f_{j_1,\ldots, j_{m}} \underbrace{z_{1,j_1}\cdot \ldots \cdot z_{\nu-1,j_{\nu-1}}z_{\nu+1,j_{\nu+1}}\cdot \ldots \cdot z_{m,j_m}}_{>0} = 0
\end{equation*}
 This implies that for all $\nu \in [m],j_\nu \in J_{\vc z_\nu}, k \in[m]\setminus\{\nu\}$ and $j_k \in [d_k]\setminus J_{\vc z_k}$, we have $f_{j_1, \ldots, j_m} = 0$, a contradiction to the irreducibility of $f$. Thus $J_{T_{\ba}(\vc z)}$ is strictly contained in $J_{\vc z}$. Using the fact that $\vc z \in \NS^d$ has at least $m$ nonzero components (because $\norm{\vc z_i}_{p_i} = 1$ for every $i \in [m]$), we get the existence of $n \in \N$ such that $T^n_{\ba}(\vc z)> 0$. Now, assume that there is $\vc z\in \NS^d_+, i \in [m]$ and $j_i \in [d_i]$ such that $\big(T_{\ba}^k(\vc z)\big)_{i,j_i} = 0$ for every $k\in\N$. Suppose by contradiction that $f$ is irreducible. Since $f \geq 0$, we must have $z_{i,j_i} = 0$. Thus, if $\kappa(\vc z) \in \N$ denotes the cardinality of $J_{\vc z}$, we have $\kappa(\vc z) > 0$. But we assumed $f$ to be irreducible and by the same arguments as above, for every $\vc x \in \NS^d_+$, if $\kappa(\vc x) > 0$, then $\kappa(\vc x)> \kappa\big(T_{\ba}(\vc x)\big)$. A contradiction to $\kappa\big(T_{\ba}^k(\vc z)\big) > 0$ for every $k \in \N$.
\item Suppose $f > 0$ and let $i \in [m],j_i \in [d_i]$ and $\vc z \in \NS^d_+$. Since $\vc z \in \NS^d_+$, for every $k \in [m]$ there exists some $j_k \in [d_k]$ such that $z_{k,j_k} > 0$. It follows that 
\begin{equation*}
0 <  f_{j_1,\ldots,j_m} z_{1,j_1}\cdot \ldots z_{i-1,j_{i-1}}z_{i+1,j_{i+1}}\cdot \ldots \cdot z_{m,j_m} \leq \partial_{i,j_i}f(\vc z) \qquad \forall j_i\in [d_i]
\end{equation*}
and thus $\sigma_{i,j_i}(\vc z)> 0$. This is true for every $i \in [m]$ and $j_i\in [d_i]$, thus $0 < \big(\alpha_1\sigma_1(\vc z), \ldots,\alpha_m\sigma_m(\vc z)\big) \leq T_{\ba}(\vc z)$. Now, suppose that there exist $l_1\in [d_1],\ldots,l_m \in [d_m]$ such that $f_{l_1,\ldots,l_m} \leq 0$. Consider the vector $\vc z$ defined for every $k \in [m]$ and $j_k \in [d_k]$ by $z_{k,j_k}=(d_1-1)^{-1/p_1}$ if $ k = 1$ and $j_1 \neq l_1$, $z_{k,j_k}=1$ if $k> 1$ and $j_k = l_k$, $z_{k,j_k}=0$ else. Then $\vc z \in \NS^d_+, z_{1,l_1} = 0$ and $\partial_{1,l_1} f(\vc z) = f_{l_1,\ldots,l_m}\leq 0$. It follows that $(T_{\ba}(\vc z))_{1,j_1} \leq 0$. \qedhere
\end{enumerate}
\end{proof}
Note that it is proved in Lemma 3.1, \cite{Fried}, that every irreducible tensor is weakly irreducible. The next example shows that the reverse implication of Proposition \ref{Characterizer_prop}, \ref{magiclema}) is not true in general.
\begin{ex}
Let $f\in \R^{2\times 2\times 2}$ be defined by $f_{1,1,1}=f_{2,2,2}=1$ and zero else. Then, for every $\vc x >0$, we have $\grad f(\vc x)=(\vc x_2 \circ \vc x_3, \vc x_1 \circ \vc x_3,\vc x_1 \circ \vc x_2)$, where $\circ$ denotes the Hadamard product. If $T_{\ba}$ is defined as in Proposition \ref{Characterizer_prop}, then $T_{\ba}(\vc x)>\alpha_0\vc x$ for every $\vc x>0$, however $f$ is not weakly irreducible.
\end{ex}
\begin{cor}\label{dual_s_pos}
Let $f\geq 0$ be a weakly irreducible tensor and $1<p_1,\ldots,p_m<\infty$, then for every $\vc z\in\sR^{d-d_i}_{++}$ and $i,k \in [m]$ with $i \neq k$, it holds $s_{i,k}(\vc z)>0$.
\end{cor}
\begin{proof}
Let $f \geq 0$ and $\vc z>0$, then $s_{i,k}(\vc z)=\sigma_k\big(\vc z_1,\ldots,\vc z_{i-1},\sigma_i(\vc z),\vc z_{i+1},\ldots,\vc z_m\big)$. It follows by Proposition \ref{Characterizer_prop}, \ref{magiclema}) that $\sigma_i(\vc z)>0$ and thus also $s_{i,k}(\vc z)>0$.
\end{proof}
\begin{cor}\label{dual_pass}
Let $f\in \R^{d_1\times\ldots\times d_m}$ with $f \neq 0$, $1<p_1,\ldots,p_m<\infty$,$(\lambda,\vc x^*)\in C^*_i,(\mu,\vc y^*)\in C^*$ and  $\Phi_i:C_i^* \to C^*$ as defined in Proposition \ref{dual_egual}. 
\begin{enumerate}[i)]
\item If $(\lambda,\vc x^*)\geq 0$ and $f \geq 0$, then $\Phi_i(\lambda,\vc x^*) \geq 0$. \label{phi_neg1}
\item If $(\lambda,\vc x^*)> 0$ and $f \geq 0$ is weakly irreducible, then $\Phi_i(\lambda,\vc x^*) > 0$.\label{phi_neg2}
\item If $(\mu,\vc y^*)\geq 0$, then $\Phi_i^{-1}(\mu,\vc y^*) \geq 0$.\label{phi_neg3}
\item If $(\mu,\vc y^*)>0$, then $\Phi_i^{-1}(\mu,\vc y^*)>0$. \label{phi_neg4}
\end{enumerate}
\end{cor}
\begin{proof}
\ref{phi_neg1}) follows from Proposition \ref{Characterizer_prop}, \ref{fgeq0}). \ref{phi_neg2}) follows from Proposition \ref{Characterizer_prop}, \ref{magiclema}). Finally \ref{phi_neg3}) and \ref{phi_neg4}) follow from the definition of $\Phi_i$. 
\end{proof}
As we need the existence of some $\vc x^*>0$ such that $Q(\vc x^*)=\norm{f}_{p_1,\ldots,p_m}$ for the Collatz-Wielandt analysis, we provide here a few conditions on $f$ and $1<p_1,\ldots,p_m<\infty$ in order to guarantee it.
\begin{lem}\label{exi_max_sing}
Let $f \geq 0$ and $1<p_1,\ldots,p_m<\infty$, then there exists some singular vector $\vc x^* \in\sR^d_+$ of $f$ such that $\norm{f}_{p_1,\ldots,p_m}=Q(\vc x^*)$.
\end{lem}
\begin{proof} 
Existence of $\vc x^* \in \NS^d$ such that $f(\vc x^*)= \norm{f}_{p_1,\ldots,p_m}$ follows from the proof of Proposition \ref{boundedspectrum}. Since $f(\vc x)\leq f(|\vc x|)$ and $S(|\vc x|)=S(\vc x)$ for every $\vc x \in \sR^d$, $\tilde{\vc x}^* \coloneqq |\vc x^*| \geq 0$ is also a maximizer of $Q$, i.e. $Q(\tilde{\vc x}^*)=\norm{f}_{p_1,\ldots,p_m}$. By definition, the singular vectors of $f$ are the critical points of $Q$ and thus $\tilde{\vc x}^*$ is a singular vector of $f$ associated to the singular value $\norm{f}_{p_1,\ldots,p_m}$.
\end{proof}
\begin{thm}\label{strictpos_irr}
Let $1<p_1,\ldots,p_m<\infty$ and $f\geq 0$ an irreducible tensor. If $\vc x\in \NS^{d}_+$ is a singular vector of $f$, then $\vc x\in\NS_{++}^d$. Moreover, there exists a singular vector $\vc x^*\in\NS^d_{++}$ of $f$ such that $Q(\vc x^*) = \norm{f}_{p_1,\ldots,p_m}$.
\end{thm}
\begin{proof}
By Equation \eqref{eigensys} we have $\sigma_i(\vc x)=\lambda^{p_i'-1} \vc x$ for every $i\in[m]$, where $\lambda=Q(\vc x)$ is the singular value associated to $\vc x$. From Proposition \ref{Characterizer_prop} we know that, with $\ba = \frac{1}{2}\big(1,\lambda^{1-p_1'},\ldots,\lambda^{1-p_m'}\big)$, there is some $n \in \mathbb{N}$ such that $0<T^{n}_{\ba}(\vc x)=\vc x$. Existence of $\vc x^*>0$ follows from Lemma \ref{exi_max_sing} and the discussion above.
\end{proof}
Using a similar argument as Friedland et al. (Theorem 3.3, \cite{Fried}), we show that if there is $i\in [m]$ such that $(m-1)p_i' \leq p_k$ for every $k\in[m]\setminus\{i\}$ and $f\geq 0$ is weakly irreducible, then there exists a strictly positive singular vector associated to the maximal singular value of $f$. First, we recall a Theorem proved by Gaubert and Gunawardena.
\begin{thm}[Theorem 2, \cite{Gaubert}]\label{Gaub}
Let $F:\R^n_{++}\to \R^n_{++}$. For $t >0$ and $J \subset [n]$ denote by $\vc u^J(t) = (u^J_1(t),\ldots,u^J_n(t))\in\R^n_{++}$ the following vector: $u^J_i(t) = 1+t$ if $i \in J$ and $u^J_i(t) = 1$ if $i \notin J$. The di-graph $\mathcal{G}(F)=\big([n], \mathcal{E}(F)\big)$ associated to $F$ is defined as follow: $(k,l) \in \mathcal{E}(F)$ if and only if $\lim_{t\to \infty} F_k\big(\vc u^{\{l\}}(t)\big) = \infty$. If $F(\vc x) \leq F(\vc y)$ for every $0<\vc x\leq\vc y$, $F$ is positively $1$-homogeneous and $\mathcal{G}(F)$ is strongly connected, then there exists $\lambda > 0$ and $\vc x \in \R^n_{++}$ such that $F(\vc x) = \lambda \vc x$.
\end{thm}
\begin{lem}\label{homos}
Let $1<p_1,\ldots,p_m<\infty,f\in\R^{d_1\times \ldots \times d_m}$ and $i,k\in [m]$ with $i\neq k$, then
\begin{equation*}
s_{i,k}(\theta_1\vc x_1,\ldots, \theta_m\vc x_m) =\left(\prod_{l\in\msi}\theta_l\right)^{p_i'(p_k'-1)} \theta_k^{1-p_k'} \ s_{i,k}(\vc x) \qquad \forall \theta_1,\ldots,\theta_m \in \R_{++},\vc x \in \sR^{d-d_i}.
\end{equation*}
\end{lem}
\begin{proof}
Follows from a straightforward computation.
\end{proof}
\begin{thm}\label{weak_strictpos}
Let $1<p_1,\ldots,p_m<\infty$ and $f\geq 0$ be a weakly irreducible tensor. Suppose that there is $i \in [m]$ such that $(m-1)p_i' \leq p_k$ for every $k\in\msi$, then $f$ has a strictly positive singular vector $\vc x^* >0$ so that $Q(\vc x^*)=\norm{f}_{p_1,\ldots,p_m}$.
\end{thm}
\begin{proof}
Let $A_i\colon\sR^{d-d_i}_{+}\to\sR^{d-d_i}_{+}$ be defined by $A_i(\vc x) \coloneqq \big(A_{i,1}(\vc x),\ldots,$ $A_{i,i-1}(\vc x),A_{i,i+1}(\vc x),\ldots $, $A_{i,m}(\vc x)\big)$ with
\begin{equation*}
A_{i,k,j_k}(\vc x) \coloneqq \bigl x_{k,j_k}^{p_{\max}-p_k}\norm{\vc x_k}_{p_k}^{p_k-p_i'(m-1)} \psi_{p_k}\big(s_{i,k,j_k} (\vc x)\big)\bigr^{\frac{1}{p_{\max}-1}} \qquad \forall k \in [m]\setminus\{i\} \text{ and } j_k \in [d_k],
\end{equation*} 
where $p_{\max} \coloneqq\max_{k \in [m] \setminus \{i\}} p_k$. We say that a vector $\vc x\in\sR^{d-d_i}$ is an eigenvector of $A_i$ if $\vc x \neq 0$ and there exists $\lambda\in\R$ such that $A_i(\vc x) = \lambda \vc x$. Suppose that we can prove the following.
\begin{enumerate}[i)]
\item $A_i|_{\sR^{d-d_i}_{++}}$ satisfies all assumptions of Theorem \ref{Gaub}.\label{step1}
\item If $\vc x^*\in\sR^{d-d_i}_{++}$ and $\vc y \in\sR^{d-d_i}_{+}\setminus\{0\}$ are such that $A_i(\vc x^*)= \lambda\vc x^*$ and $A_i(\vc y)= \mu\vc y$ then $\mu\leq\lambda$.
\label{step2}
\item If $\vc x\in \sR^{d-d_i}_{++}$ is an eigenvector of $A_i$ associated to a strictly positive eigenvalue $\lambda$, then it is a critical point of $Q_i$ associated to the critical value $\lambda^{(p_{\max}-1)/p_i'}$.\label{step3}
\item If $\vc x \in \NS^{d-d_i}_+$ is a critical point of $Q_i$ associated to $\lambda$, then $A_i(\vc x)=\lambda^{p_i'/(p_{\max}-1)}\vc x$.\label{step4}
\end{enumerate}
Using \ref{step1}), we may apply Theorem \ref{Gaub} and get the existence of $\vc x^* >0$ such that $A_i(\vc x^*) =\lambda \vc x^*$. \ref{step2}) ensures that the associated eigenvalue $\lambda$ is maximal. By \ref{step3}), we know that $\vc x^*>0$ is a critical point of $Q_i$. Now, suppose by contradiction that $\lambda^{(p_{\max}-1)/p_i'}=Q_i(\vc x^*)<\norm{f}_{p_1,\ldots,p_m}$. Lemma \ref{exi_max_sing} and Proposition \ref{dual_egual} imply the existence of $\vc y^*\geq 0$ such that $Q_i(\vc y^*)=\norm{f}_{p_1,\ldots,p_m}$. By \ref{step4}), we know that $\vc y^*$ is an eigenvector of $A_i$ associated to $\norm{f}_{p_1,\ldots,p_m}^{p_i'/(p_{\max}-1)}>\lambda$, a contradiction. Finally, since $f$ is weakly irreducible we know from Corollary \ref{dual_pass} that $(\norm{f}_{p_1,\ldots,p_m},\tilde{\vc x}^*) \coloneqq \Phi_i\big(\norm{f}_{p_1,\ldots,p_m},\vc x^*\big)$ is such that $\tilde{\vc x}^*\in\sR^d$ is a strictly positive singular vector of $f$ with $Q(\tilde{\vc x}^*)=\norm{f}_{p_1,\ldots,p_m}$. Now, we prove $\ref{step1})-\ref{step4})$.
\begin{enumerate}[i)]
\item The fact that $A_i$ is positively $1$-homogeneous follows from Lemma \ref{homos}. If $0 \leq \vc x \leq \vc y$, then $A_i(\vc x) \leq A_i(\vc y)$ follows from $p_{\max}-p_k\geq 0, p_k-p_i'(m-1) \geq 0$ and Lemma \ref{order_preserving}. In order to prove the remaining property, that $\mathcal{G}(A_i)$ is strongly connected, we use the fact that $f$ is weakly irreducible. So, let $\mathcal{G}(A_i)=\big(\mathcal{V},\mathcal{E}(A_i))$ and $G(f)=\big(V,E(f)\big)$ (see Definition \ref{def_irr} and Theorem \ref{Gaub}) with $\mathcal{V}\coloneqq (\{1\}\times[d_1])\cup\ldots\cup(\{i-1\}\times[d_{i-1}])\cup(\{i+1\}\times[d_{i+1}])\cup\ldots\cup(\{m\}\times[d_m])$ and $V\coloneqq (\{1\}\times[d_1])\cup\ldots\cup(\{m\}\times[d_m])$. 
We show that for every $k,l\in\msi$, if $\big((k,j_k),(l,j_l)\big) \in E(f)$ then $\big((k,j_k),(l,j_l)\big)\in \mathcal{E}(A_i)$ and if $\big((k,j_k),(i,j_i)\big),\big((i,j_i),(l,\nu_l)\big)\in E(f),$ then $\big((k,j_k),(l,\nu_l)\big) \in \mathcal{E}(A_i)$. 
Since $G(f)$ is an undirected connected graph ($f$ is weakly irreducible) this would imply that $\mathcal{G}(A_i)$ is strongly connected. By definition, $\big((k,j_k),(l,j_l)\big)\in E(f)$ implies the existence of $j_s \in [d_s]$ for $s\in [m]\setminus\{k,l\}$ such that $f_{j_1,\ldots,j_m} >0$ and if $\vc u^{\{(l,j_l)\}}(t)> 0$ is defined as in Theorem \ref{Gaub}, then $\partial_{k,j_k} f\big(\vc u^{\{(l,j_l)\}}(t)\big) \to \infty $ as $t \to \infty$. It follows by Proposition \ref{Characterizer_prop}, \ref{magiclema}) and the definition of $s_{i,k}$ that $A_{i,k,j_k}\big(\vc u^{\{(l,j_l)\}}(t)\big) \to \infty$ as $t \to \infty$. 
Now, suppose that $\big((k,j_k),(i,j_i)\big),\big((i,j_i),(l,\nu_l)\big)\in E(f)$, then there exists $j_s \in [d_s]$ and $\nu_t \in [d_t]$ for $s \in [m]\setminus\{k,i\}$ and $t \in [m]\setminus\{l,i\}$ such that $f_{j_1,\ldots,j_m} >0$ and $f_{\nu_1,\ldots,\nu_{i-1},j_i,\nu_{i+1},\ldots,\nu_m} >0$. Again, let $\vc u^{\{(l,\nu_l)\}}(t)> 0$ be defined as in Theorem \ref{Gaub}, then $f_{\nu_1,\ldots,\nu_{i-1},j_i,\nu_{i+1},\ldots,\nu_m} >0$ implies that $\partial_{i,j_i} f\big(\vc u^{\{(l,\nu_l)\}}(t)\big)\to \infty$ as $t \to \infty$ and $f_{j_1,\ldots,j_m} >0$ implies that 
\begin{equation*}
\partial_{k,j_k}f\bigg(\vc u^{\{(l,\nu_l)\}}_1(t), \ldots, \vc u^{\{(l,\nu_l)\}}_{i-1}(t),\psi_{p_i'}\bigl\grad_{i} f\big(\vc u^{\{(l,\nu_l)\}}(t)\big)\bigr, \vc u^{\{(l,\nu_l)\}}_{i+1}(t),\ldots,\vc u^{\{(l,\nu_l)\}}_m(t)\bigg) \to \infty,
\end{equation*}
as $t \to \infty$. It follows that $A_{i,k,j_k}\big(\vc u^{\{(l,\nu_l)\}}(t)\big) \to \infty $ as $t \to \infty$, i.e. $\big((k,j_k),(l,\nu_l)\big)\in \mathcal{E}(A_i)$. Thus $A_i$ fulfills all assumptions of Theorem \ref{Gaub}.
\item This result follows from Lemma 3.3 in \cite{Nussbaum}. We reproduce the proof here for convenience of the reader. Let $\vc x^*\in\sR^{d-d_i}_{++}$ and $\vc y \in\sR^{d-d_i}_{+}\setminus\{0\}$ be such that $A_i(\vc x^*)=\lambda\vc x^*$ and $A_i(\vc y)= \mu\vc y$. Set $\theta\coloneqq\min\left\{\frac{x^*_{k,j_k}}{y_{k,j_k}}\mid y_{k,j_k}>0, k \in [m]\setminus\{i\},j_k \in [d_k]\right\}$, then $\theta>0$ and $\theta\vc y\leq \vc x^*$. We observed in \ref{step1}) that $A_i$ is positively $1$-homogeneous and $0 \leq\vc x\leq\vc y$ imply $A_i(\vc x)\leq A_i(\vc y)$. It follows that for every $n\in\N$, we have $\theta \mu^n \vc y = \theta A^n_i(\vc y) = A_i^n(\theta \vc y) \leq A^n_i(\vc x^*) =\lambda^n \vc x,$
where $A^k_i(\vc z) \coloneqq A^{k-1}_i\big(A_i(\vc z)\big)$ for $k\in\N, k>1$ and $\vc z\geq 0$. This shows that if $\frac{\lambda}{\mu}<1$, then $\theta \vc y \leq \lim_{n\to\infty} \left(\frac{\lambda}{\mu}\right)^n\vc x=0$, a contradiction to $\vc y\in\sR^{d-d_i}_{+}\setminus\{0\}$.
\item Let $\vc x> 0$ be an eigenvector of $A_i$ associated to the eigenvalue $\lambda>0$. Note that $A_i (\vc x) = \lambda \vc x$ implies
\begin{equation}\label{refmilieu}
 \lambda^{p_{\max}-1} x_{k,j_k}^{p_k-1}= \norm{\vc x_k}_{p_k}^{p_k-p_i'(m-1)} \psi_{p_k}\big(s_{i,k,j_k}(\vc x)\big) \qquad \forall k \in [m]\setminus\{i\}, j_k \in [d_k].
\end{equation}
Multiplying this equation by $x_{k,j_k}$ and summing over $j_k\in [d_k]$ shows 
\begin{equation*}
\lambda^{p_{\max}-1} \norm{\vc x_{k}}_{p_k}^{p_k} = \norm{\vc x_k}_{p_k}^{p_k-p_i'(m-1)} \ps{\psi_{p_k}\big(s_{i,k}(\vc x)\big) }{\vc x_k} =\norm{\vc x_k}_{p_k}^{p_k-p_i'(m-1)} \norm{\grad_i f(\vc x)}_{p_i'}^{p_i'}\quad \forall k \in\msi,
\end{equation*}
where we used $\ps{\psi_{p_k}\big(s_{i,k}(\vc x)\big) }{\vc x_k}=f\big(\vc x_1,\ldots,\vc x_{i-1},\sigma_i(\vc x),\vc x_{i+1},\ldots,\vc x_m\big)=\ps{\psi_{p_i'}\big(\grad_if(\vc x)\big)}{\grad_if(\vc x)}$.
 Thus $\norm{\vc x_k}_{p_k}^{p_i'(m-1)} = \norm{\grad_i f(\vc x)}_{p_i'}^{p_i'}\lambda^{1-p_{\max}}$ for every $k\in\msi$, i.e. $\norm{\vc x_1}_{p_1}=\ldots=\norm{\vc x_m}_{p_m}\eqqcolon\alpha$ and we have $\alpha^{-1} \vc x \in \NS^{d-d_i}$. Now, substituting $\vc x$ by $\tilde{\vc x}= \alpha^{-1}\vc x$ in Equation \eqref{refmilieu} and composing by $\psi_{p_k'}$ shows $\lambda^{(p_{\max}-1)(p_k'-1)}\tilde{\vc x}_{k,j_k}=s_{i,k,j_k}(\tilde{\vc x})$. By Equation \eqref{dual_eigensys} it follows that $\tilde{\vc x}$ is a critical point of $Q_i$. Since the critical points of $Q_i$ are scale invariant, $\vc x=\alpha\tilde{\vc x}$ is also a critical point of $Q_i$ and $Q_i(\vc x)= \lambda^{(p_{\max}-1)/p_i'}$.
 \item Suppose that $\vc x \in \NS^{d-d_i}_+$ is a critical point of $Q_i$ and let $\lambda>0$ its associated critical value. Since $\vc x\in \NS^{d-d_i}$ is a singular vector of $f$, by Equation \eqref{dual_eigensys} we have $s_{i,k}(\vc x) = \lambda^{p_i'(p_k'-1)} \vc x_k$  for every $k\in\msi$. It follows that $
 A_{i,k,j_k}(\vc x) =  \bigl x_{k,j_k}^{p_{\max}-p_k}\psi_{p_k}\big( \lambda^{p_i'(p'_k-1)} x_{k,j_k}\big)\bigr^{1/(p_{\max}-1)} = \lambda^{p_i'/(p_{\max}-1)} x_{k,j_k},$
 i.e. $\vc x$ is an eigenvector of $A_i$ associated to the eigenvalue $\lambda^{p_i'/(p_{\max}-1)}>0$.\qedhere
\end{enumerate}
\end{proof}
\section{Collatz-Wielandt analysis and proof of the main Theorem}\label{thethm}
For $i \in [m]$, consider $\mbd_i,\Mbd_i \colon \NS^{d-d_i}_{++}\to \R_{++}$, the Collatz-Wielandt ratios defined as follow:
\begin{equation}\label{gammadef}
\mbd_i(\vc x) \coloneqq \prod_{k\in\msi}\min_{j_k \in [d_k]}\left(\frac{s_{i,k,j_k}(\vc x)}{x_{k,j_k}}\right)^{p_k-1}\ \text{ and }\  \Mbd_i(\vc x)\coloneqq\prod_{k\in\msi}\max_{j_k \in [d_k]}\left(\frac{s_{i,k,j_k}(\vc x)}{x_{k,j_k}}\right)^{p_k-1}.
\end{equation}
The next lemma is useful to prove the uniqueness of the strictly positive singular vector of $f$ in our Perron-Frobenius Theorem.
\begin{lem}\label{uniweak}
Let $f\geq 0$ be a weakly irreducible tensor, $i \in [m]$ and $\vc x,\vc y \in \sR^{d-d_i}$. Assume $0< \vc x \leq \vc y$ and $\vc x \neq \vc y$, furthermore consider the following sets for $\nu \in [m]\setminus \{i\}$: $
J_{\nu} \coloneqq \{j_\nu \in [d_\nu]\mid x_{\nu,j_\nu} = y_{\nu,j_\nu}\},   J_{i}\coloneqq \{j_i \in [d_i]\mid\sigma_{i,j_i}(\vc x)=\sigma_{i,j_i}(\vc y)\},  I_\nu \coloneqq \{j_\nu\in [d_\nu]\mid s_{i,\nu,j_\nu}(\vc x)=s_{i,\nu,j_\nu}(\vc y)\},  
I_i \coloneqq J_i,   J \coloneqq \big(\{1\}\times J_1\big) \cup \ldots \cup \big(\{m\}\times J_m\big)$ and $I \coloneqq \big(\{1\}\times I_1\big) \cup \ldots \cup \big(\{m\}\times I_m\big)$.
Then $I \subset J$ and if $J \neq \emptyset$ we have $I \neq J$.
\end{lem}
\begin{proof}
First, we prove $I \subset J$. If $I = \emptyset$, there is nothing to prove, so suppose $I \neq \emptyset$. Let $(l,j_l) \in I$ with $l\in\msi$, weak irreducibility of $f$ implies the existence of $j_k\in[d_k]$ for each $k\in[m]\setminus\{l\}$ such that $f_{j_1,\ldots,j_m}>0$. So, $s_{i,l,j_l}(\vc x)=s_{i,l,j_l}(\vc y)$ implies $\sigma_{i,j_i}(\vc x)=\sigma_{i,j_i}(\vc y)$ because $0<x_{k,j_k}\leq y_{k,j_k}$ for every $k\in[m]\setminus\{l,i\}$. Now, $\sigma_{i,j_i}(\vc x)=\sigma_{i,j_i}(\vc y)$ implies $x_{l,j_l}=y_{l,j_l}$ since $0<x_{k,j_k}\leq y_{k,j_k}$ for every $k\in[m]\setminus\{l,i\}$. Thus $(l,j_l)\in J$ and we have $I\subset J$.
Now, suppose $J \neq \emptyset$, then $\vc x\neq \vc y$ implies the existence of $l\in[m]\setminus\{i\}$ and $j_l\in[d_l]$ such that $x_{l,j_l}<y_{l,j_l}$. Since $f$ is weakly irreducible, there exists $j_k\in[d_k]$ for each $k\in[m]\setminus\{l\}$ such that $f_{j_1,\ldots,j_m}>0$. It follows that $\sigma_{i,j_i}(\vc x)< \sigma_{i,j_i}(\vc y)$ and thus $I_i\neq [d_i]$ as well as $s_{i,k,j_k}(\vc x)<s_{i,k,j_k}(\vc y)$ for $k\in[m]\setminus\{i\}$. Hence, $I_k\neq [d_k]$ for every $k\in [m]$. Now, on one hand, if there is $k \in [m]$ such that $J_k=[d_k]$, then $I_k \neq J_k$ and the proof is done. On the other hand, if there is $k \in [m]$ such that $J_k=\emptyset$, then $I_l=\emptyset$ for every $l\in\msi$ and the proof is done. Finally, assume that $J_k\notin\big\{[d_k],\emptyset\big\}$ for every $k\in [m]$.
Suppose by contradiction that $I=J$. Weak irreducibility of $f$ implies the existence of a vertex between $J$ and $V\setminus J$ in the graph $G(f)=\big(V,E(f)\big)$, i.e. there exists $\nu,\mu \in [m],\nu\neq\mu, j^*_\nu\in J_\nu, j^*_\mu\in [d_\mu]\setminus J_\mu$ and $j^*_k\in [d_k]$ for each $k \in [m]\setminus\{\nu,\mu\}$ such that $f_{j^*_1,\ldots,j^*_m}>0$. 
If $\nu \neq i$, $s_{i,\nu,j^*_\nu} (\vc y)=s_{i,\nu,j^*_\nu}(\vc x)$ follows from $\nu \in J_\nu = I_\nu$ and thus $y_{\mu,j^*_\mu} = x_{\mu,j^*_\mu}$, a contradiction to $j^*_\mu\in [d_\mu]\setminus J_\mu$. If $\nu = i$, the equality $\sigma_{i,j^*_i}(\vc y) =\sigma_{i,j^*_i}(\vc x)$ implies the same contradiction.
\end{proof}
Note that the assumption $0< \vc x$ in Lemma \ref{uniweak} can't be replaced by $0\leq \vc x, S_i(\vc x) \neq 0$ as shown in the following example.
\begin{ex}\label{cex_lem}
Let  $1<p_1,p_2,p_3<\infty,i = 3$ and $f\in\R^{2\times 2\times 2}$ the nonnegative weakly irreducible tensor defined by $f_{1,1,1}=f_{1,2,1}=f_{2,2,2}=1$ and $f_{j_1,j_2,j_3}=0$ else. Let $\vc x \coloneqq \big((1,0),(1,0)\big)$ and $\vc y \coloneqq \big((1,1),(1,0)\big) \in \R^{(6-2)}$, then $0\leq \vc x \leq \vc y,\vc x\neq\vc y$ and $S_i(\vc x)=\norm{(1,0)}_{p_1}\norm{(1,0)}_{p_2}=1$. However, $s_{3,1}(\vc x)=(1,0)=s_{3,1}(\vc y)$ and $s_{3,2}(\vc x)=(1,1)=s_{3,2}(\vc y)$, thus, if $I$ and $J$ are defined as in Lemma \ref{uniweak}, we get $J \subsetneq I$.
\end{ex}
\begin{thm}\label{gen_PF}
Let $f$ be a nonnegative weakly irreducible tensor and $1 < p_1,\ldots,p_m < \infty$ are such that there exists $i \in [m]$ with $(m-1)p_i' \leq p_k$ for every $k\in \msi$, therefore there exists $\vc x^* \in \NS^{d-d_i}_{++}$ such that $Q_i(\vc x^*)=\norm{f}_{p_1,\ldots,p_m}$. Let $\mbd_i,\Mbd_i$ as in Equation \eqref{gammadef}, then for every $\vc z \in \NS^{d-d_i}_{++}$ we have $
\mbd_i(\vc z)\leq \norm{f}_{p_1,\ldots,p_m}^{p_i'(m-1)} \leq \Mbd_i(\vc z)$ and equality holds if and only if $\vc x^* = \vc z$.
\end{thm}
\begin{proof}
First of all, note that, by Theorem \ref{weak_strictpos}, we know that there exists some singular vector $\tilde{\vc x}^*\in \NS^d$ of $f$ such that $\tilde{\vc x}^* > 0$ and $f(\tilde{\vc x}^*)=\norm{f}_{p_1,\ldots,p_m}$. Using the bijection $\Phi_i$ of Proposition \ref{dual_egual} and Corollary \ref{dual_pass}, we get the existence of $\vc x^* \in \NS^{d-d_i}$ such that $\vc x^* > 0$ and $Q_i(\vc x^*)=\norm{f}_{p_1,\ldots,p_m}$. 
From Equation \eqref{dual_eigensys} and $\vc x^* \in \NS^{d-d_i}_{++}$ follows
$
\bigl\frac{s_{i,k,j_k}(\vc x^*)}{x^*_{k,j_k}}\bigr^{p_k-1}= \norm{f}_{p_1,\ldots, p_m}^{p_i'}$ for all $k \in\msi$ and every $j_k\in [d_k].$
In particular, this implies that $\mbd_i(\vc x^*)= \norm{f}_{p_1,\ldots,p_m}^{p_i'(m-1)}=\Mbd_i(\vc x^*)$. If $\vc z = \vc x^*$, then $\mbd_i(\vc z)=\mbd_i(\vc x^*) =\norm{f}_{p_1,\ldots,p_m}^{p_i'(m-1)}=\Mbd_i(\vc x^*)=\Mbd_i(\vc z)$ and the proof is done. Suppose that $\vc z \neq \vc x^*$ and
for $l \in\msi,$ let $
\theta_l \coloneqq \min_{j_l \in [d_l]} \frac{z_{l,j_l}}{x_{l,j_l}^*} $ and $\Theta_l \coloneqq \max_{j_l \in [d_l]}\frac{z_{l,j_l}}{x_{l,j_l}^*}$, then $\theta_l\vc x_l^*\leq \vc z_l \leq \Theta_l \vc x_l^*$ and $\theta_l = \norm{\theta_l \vc x_l^*}_{p_l} \leq \norm{\vc z_l}_{p_l}=1\leq \norm{\Theta_l \vc x_l^*}_{p_l} = \Theta_l,$ i.e $\theta_l \in ]0,1]$ and $\Theta_l \in [1,\infty[$. Observe that Lemmas \ref{order_preserving} and \ref{homos} imply
\begin{equation*}
\left(\prod_{l\in\msi}\theta_l\right)^{p_i'} \theta_k^{-1} \ \psi_{p_k}\big(s_{i,k,j_k}(\vc x^*)\big) \leq \psi_{p_k}(s_{i,k,j_k}\big(\vc z)\big)\leq \left(\prod_{l\in\msi}\Theta_l\right)^{p_i'} \Theta_k^{-1}\ \psi_{p_k}\big(s_{i,k,j_k}(\vc x^*)\big).
\end{equation*}
Moreover,
\begin{equation*}
\prod_{k\in\msi} \left(\left(\prod_{l\in\msi}\theta_l\right)^{p_i'} \theta_k^{-1}\right) =\left(\prod_{l\in\msi}\theta_l\right)^{(m-1)p_i'} \left(\prod_{k\in\msi} \theta_k^{-1} \right)=\prod_{l\in\msi} \theta_l^{(m-1)p_i'-1}
\end{equation*}
and the same holds if we replace $\theta_l$ by $\Theta_l$. The assumption $(m-1)p_i'\leq p_k$ guarantees $
\Theta_k^{(m-1)p_i'-p_k} \leq 1 \leq \theta_k^{(m-1)p_i'-p_k} $ for every $k \in\msi$. Now, if $\theta_l=1$ for every $l\in\msi$, then $\vc z_l = \vc x_l^*$ for every $l\in\msi$, a contradiction to $\vc x^*\neq\vc z$. For the same reason, we can't have $\Theta_l = 1$ for every $l\in\msi$. Hence, there is some $l,k\in [m]\setminus\{i\}$ such that $\theta_l<1$ and $\Theta_k >1$. Thus $\theta_l \vc x^*_l \neq \vc z_l$ and $\vc z_k\neq \Theta_k\vc x^*_k$. Applying Lemma \ref{uniweak} to $0<(\theta_1\vc x^*_1,\ldots, \theta_m\vc x^*_m)\leq \vc z$ and $0<\vc z \leq (\Theta_1\vc x^*_1,\ldots, \Theta_m\vc x^*_m)$, we get the existence of $\mu,\nu \in [m]\setminus\{i\},j_{\mu}^+\in [d_{\mu}]$ and $j_{\nu}^-\in [d_{\nu}]$ such that $\theta_{\mu}x^*_{{\mu},j_{\mu}^-} = z_{{\mu},j_{\mu}^-}$ and $ \Theta_{\nu}x^*_{{\nu},j_{\nu}^+}=z_{{\nu},j^+_{\nu}},$ as well as
\begin{equation}\label{strict_ineq_pf}
s_{i,{{\mu},j_{\mu}^-}}(\theta_1\vc x^*_1,\ldots,\theta_m \vc x^*_m)<
s_{i,{{\mu},j^-_{\mu}}}(\vc z) \quad \text{ and }\quad s_{i,{\nu},j_{\nu}^+}(\vc z)<
s_{i,{\nu},j^+_{\nu}}(\Theta_1\vc x^*_1,\ldots,\Theta_m \vc x^*_m).
\end{equation}
Furthermore, for each $l \in [m]\setminus\{\mu,\nu,i\}$ there exists indexes $j_{l}^+,j_{l}^-\in [d_l]$ such that $\theta_lx^*_{l,j_{l}^-} = z_{l,j^-_{l}}$ and $\Theta_lx^*_{l,j_{l}^+} = z_{l,j^+_{l}} $ by construction of $\theta_l$ and $\Theta_l$. With Lemma \ref{order_preserving}, we get
\begin{eqnarray*}
\prod_{l\in\msi}\frac{\psi_{p_l}\big(s_{i,l,j_{l}^-}(\vc z)\big)}{\psi_{p_l}(z_{l,j^-_{l}})} &>&\left(\prod_{l\in\msi} \theta_l^{(m-1)p_i'-p_l}\right)\left(\prod_{l\in\msi}\frac{\psi_{p_l}\big( s_{i,l,j_{l}^-}(\vc x^*)\big)}{\psi_{p_l}(x^*_{l,j_{l}^-})}\right)
\\&\geq &\left(\prod_{l\in\msi}\frac{\psi_{p_l}\big(s_{i,l,j_{l}^-}(\vc x^*)\big)}{\psi_{p_l}(x^*_{l,j_{l}^-})}\right) = \norm{f}_{p_1,\ldots,p_m}^{p_i'(m-1)} 
> \prod_{l\in\msi}\frac{\psi_{p_l}\big(s_{i,l,j_{l}^+}(\vc z)\big)}{\psi_{p_l}(z_{l,j_{l}^+})}, 
\end{eqnarray*}
where we used Equation \eqref{strict_ineq_pf} for the strict inequalities.
Observing that
\begin{equation*}
\Mbd_i(\vc z) \geq \prod_{l\in\msi}\frac{\psi_{p_l}\big(s_{i,l,j_{l}^-}(\vc z)\big)}{\psi_{p_l}(z_{l,j_{l}^-})} = \prod_{l\in\msi}\left(\frac{s_{i,l,j_{l}^-}(\vc z)}{z_{l,j_{l}^-}} \right)^{p_l-1} \quad  \text{ and } \quad \prod_{l\in\msi}\frac{\psi_{p_l}\big(s_{i,l,j_{l}^+}(\vc z)\big)}{\psi_{p_l}(z_{l,j_{l}^+})}\geq\mbd_i(\vc z),
\end{equation*}
shows $\Mbd_i(\vc z)>\norm{f}_{p_1,\ldots,p_m}^{p_i'(m-1)}>\mbd_i(\vc z)$. 
\end{proof}
\begin{cor}\label{unique_part}
Let $f\geq 0$ be a weakly irreducible tensor and $1 < p_1,\ldots,p_m < \infty$ are such that there exists $i \in [m]$ with $(m-1)p_i' \leq p_k$ for every $k\in \msi$, therefore there exists $\vc x^* \in \NS^{d-d_i}_{++}$ with $Q_i(\vc x^*)=\norm{f}_{p_1,\ldots,p_m}$. Then $\vc x^*$ is the unique critical point of $Q_i$ in $\NS^{d-d_i}_{++}$. Moreover, if $f$ is irreducible, then $\vc x^*$ is the unique critical point of $Q_i$ in $\NS^{d-d_i}_+$.
\end{cor}
\begin{proof}
Suppose that $\vc y^*\in \NS^{d-d_i}_{++}$ is a critical point of $Q_i$. By Equation \eqref{dual_eigensys} we know that $\left(\frac{s_{i,k,j_k}(\vc y^*)}{y^*_{k,j_k}}\right)^{p_k-1}=Q_i(\vc y^*)^{p_i'}$ for every $k\in [m]\setminus\{i\}$ and $j_k\in [d_k]$. Thus, Theorem \ref{gen_PF} implies $\vc y^* = \vc x^*$ since $Q_i(\vc y^*)^{p_i'(m-1)}=\mbd_i(\vc y^*)\leq\norm{f}_{p_1,\ldots,p_m}^{p_i'(m-1)}\leq\Mbd_i(\vc y^*)=Q_i(\vc y^*)^{p_i'(m-1)}$. Now, suppose that $f$ is irreducible. By Theorem \ref{strictpos_irr} we know that every nonnegative singular vector of $f$ is strictly positive. Since $\vc x^*$ is the unique singular vector of $f$ in $\NS^{d}_{++}$, it is also the unique singular vector in $\NS^{d}_+$. 
\end{proof}
 As shown in the next example, there are weakly irreducible tensors with nonnegative singular vectors.
 \begin{ex}
 Let $f\in \R^{2\times 2 \times 2}$ be the tensor of Example \ref{cex_lem} and $1<p_1,p_2,p_3<\infty$ such that there is $i\in[3]$ with $2p_i'\leq p_k$ for every $k\in\msi$. Then, by Theorem \ref{Gaub}, there exists a singular vector $\vc x^*\in\NS^{d}_{++}$ of $f$. However, note that $\vc y^*\coloneqq \big((0,1),(0,1),(0,1)\big)\in\NS^{d}_+$ is also a singular vector of $f$.
 \end{ex}
 Now, we discuss the assumptions made on $p_1,\ldots,p_m$ in Theorem \ref{gen_PF}. 
\begin{rmq}\label{pi_prop}
Let $1 < p_1,\ldots,p_m < \infty$ and $i\in [m]$, then
\begin{equation*}
(m-1)p_i' \leq \min_{k\in [m]\setminus\{i\}} p_k \iff m-1 \leq (p_i-1)\left(\min_{k \in [m]\setminus\{i\}} p_k-(m-1)\right).
\end{equation*}
Furthermore, if there exists $i \in [m]$ with $(m-1)p_i'\leq p_k$ for every $k \in\msi$, then at least one of the following condition is satisfied
\begin{enumerate}[a)]
\item $(m-1)p'_{l} \leq p_k,$ for every $k \in\msi$, where $p_{l} = \min_{k \in [m]} p_k$,
\item $(m-1)p'_{n} \leq p_k,$ for every $k \in\msi$, where $p_{n} = \max_{k \in [m]} p_k$.
\end{enumerate}
Moreover, note that if $(m-1)p_i' \leq p_k$ for all $k\in\msi$ and $m\geq 3$, then $2 \leq (m-1) \leq (m-1)p_i' \leq p_k$ for every $k\in \msi$. In the case $m=2$, we may also always choose $i \in [m]$ such that $p_k \geq 2$ for $k \in[2]\setminus\{i\}$ because $p_1'\leq p_2$ is equivalent to $p_2'\leq p_1$ which is true if and only if $1 \leq (p_1-1)(p_2-1)$. Thus we recover the matrix case as analyzed by Boyd \cite{Boyd}. Note however that our result implies that the strictly positive singular vector is unique. Bhaskara et al. \cite{Bhaskara} proved the uniqueness of the strictly positive singular vector but only for the case of strictly positive matrix.
\end{rmq}
\begin{proof}[Proof of Theorem \ref{gen_sum_PF}]
The min-max characterization follows from Theorem \ref{gen_PF}. The uniqueness par follows from Proposition \ref{dual_egual} and Corollary \ref{unique_part}.
\end{proof}

\section{Relation with other tensor spectral problems}\label{Other_spectra_pb}
It is well-known that if a matrix is symmetric, then its eigenvectors and singular vectors coincide up to sign. This observation can be extended for tensors as shown in the following proposition.
\begin{prop}\label{relation_prop}
Let $q_1,\ldots,q_k\in\N\setminus\{0\}$ such that $q_1+\ldots+q_k=m$ and $f \in\R^{d_1\times \ldots\times d_m}$ with $d_1 = \ldots = d_{q_1}=\tilde{d}_1, d_{(q_1+1)} = \ldots = d_{(q_1+q_2)}=\tilde{d}_2,\ldots,d_{(q_1+\ldots+q_{k-1}+1)}= \ldots = d_{(q_1+\ldots +q_k)}=\tilde{d}_k$. Suppose that $f$ is partially symmetric in the sense that
$$f_{j_1,\ldots,j_m}=f_{\sigma_1(j_1,\ldots,j_{q_1}), \sigma_2(j_{(q_1+1)},\ldots,j_{(q_1+q_2)}),\ldots,\sigma_{k}(j_{(q_1+\ldots+q_{k-1}+1)},\ldots,j_{(q_1+\ldots+q_k)})}$$
where, for $i = 1,\ldots, k$, $\sigma_i$ is any permutation of $q_i$ elements. Then, every solution $(\lambda,\vc x_1,\ldots,\vc x_k) \in \big(\R\setminus\{0\}\big)\times \R^{\tilde{d}_1}\times \R^{\tilde{d}_2}\times\ldots\times\R^{\tilde{d}_m}$ of the problem
\begin{equation}\label{singeig_pb}
\left\{ \begin{array}{l l} \grad_{(q_{1}+\ldots+q_{i-1}+1)}f(\vc x_1,\ldots,\vc x_1,\vc x_2,\ldots,\vc x_2,\ldots,\vc x_k,\ldots,\vc x_k) = \lambda \psi_{\tilde{p}_i}(\vc x_i) &\qquad i = 1,\ldots,k,\\
\norm{\vc x_1}_{\tilde{p}_1}=\ldots=\norm{\vc x_k}_{\tilde{p}_k}=1,\end{array}\right.
\end{equation}
where $1<\tilde{p}_1,\ldots,\tilde{p}_{k}<\infty$, induces a $\ell^{p_1,\ldots,p_m}$ singular vector of $f$ with $p_1 = \ldots = p_{q_1}=\tilde{p}_1, p_{(q_1+1)} = \ldots = p_{(q_1+q_2)}=\tilde{p}_2,\ldots,p_{(q_1+\ldots+q_{k-1}+1)}= \ldots = p_{(q_1+\ldots +q_k)}=\tilde{p}_k$. In particular, if $f$ and $p_1,\ldots,p_m$ satisfy the conditions of Theorem \ref{gen_sum_PF}, then there is a unique strictly positive solution $(\lambda^*,\vc x^*_1,\ldots,\vc x^*_k)$ to Problem \eqref{singeig_pb}. If $f$ is irreducible, this solution is also the only nonnegative solution. Moreover, if $(\lambda^*,\tilde{\vc x}^*_1,\ldots,\tilde{\vc x}^*_m)$ is the unique strictly positive $\ell^{p_1,\ldots,p_m}$-singular vector of $f$, then $(\lambda,\tilde{\vc x}^*_1,\tilde{\vc x}^*_{(q_1+1)},\ldots,\tilde{\vc x}^*_{(q_1+\ldots+q_{k-1}+1)})$ is the unique strictly positive solution of \eqref{singeig_pb}.
\end{prop}
\begin{proof}
For $(\vc y_1,\ldots,\vc y_k)\in\R^{\tilde{d}_1}\times \R^{\tilde{d}_2}\times\ldots\times\R^{\tilde{d}_m}$, consider the injective map
$$\xi(\vc y_1,\ldots,\vc y_k) \coloneqq (\underbrace{\vc y_1,\ldots,\vc y_1}_{q_1 \text{ times}},\underbrace{\vc y_2\ldots,\vc y_2}_{q_2 \text{ times}},\ldots,\underbrace{\vc y_k,\ldots,\vc y_k}_{q_k \text{ times}}) \in \R^{d_1}\times \ldots\times \R^{d_m},$$
and the surjective map 
$$\zeta(\vc z_1,\ldots,\vc z_m)\coloneqq (\vc z_1,\vc z_{(q_1+1)},\ldots,\vc z_{(q_1+\ldots+q_{k-1}+1)})\in\R^{\tilde{d}_1}\times \R^{\tilde{d}_2}\times\ldots\times\R^{\tilde{d}_m} $$
defined for $(\vc z_1,\ldots,\vc z_m)\in\R^{d_1}\times\ldots\times\R^{d_m}$. Note that $\zeta\big(\xi(\vc y)\big)=\vc y$ for every $\vc \in \R^{\tilde{d}_1}\times\ldots\times\R^{\tilde{d}_m}$.
If $(\lambda,\vc x_1,\ldots,\vc x_k)$ is a solution of \eqref{singeig_pb}, then the partial symmetry of $f$ and the definition of $\xi$ imply
\begin{align*}
&\grad_{(q_1+\ldots+q_{i-1}+l)}f\big(\xi(\vc x_1,\ldots,\vc x_k)\big)=\grad_{(q_1+\ldots+q_{i-1}+1)}f\big(\xi(\vc x_1,\ldots,\vc x_k)\big)=\lambda\ \psi_{\tilde{p}_i}(\vc x_i)\\&\qquad =\lambda\ \psi_{p_i}\big(\xi(\vc x_1,\ldots,\vc x_k)_{(q_1+\ldots+q_{i-1}+l)}\big) \qquad \forall l = 1,\ldots,q_i, i = 1,\ldots,k.
\end{align*}
This shows that $\xi(\vc x_1,\ldots,\vc x_k)$ is a $\ell^{p_1,\ldots,p_m}$-singular vector of $f$ associated to the singular value $\lambda$. On the other hand, if $\xi(\vc x_1,\ldots,\vc x_k)$ is a $\ell^{p_1,\ldots,p_m}$-singular vector of $f$, then, from Equation \eqref{eigensys}, we know that $(\vc x_1,\ldots,\vc x_k)$ is a solution to Problem \eqref{singeig_pb}. Now, suppose that $f$ and $p_1,\ldots,p_m$ satisfy the conditions of Theorem \ref{gen_sum_PF}. The existence of a strictly positive solution $(\lambda^*,\vc x^*_1,\ldots,\vc x^*_k)$ to Problem \eqref{singeig_pb}, can be shown in the same way as Theorem \ref{weak_strictpos} by considering the map $\tilde{A}_{i}\colon\R^{\tilde{d}_1}\times\ldots\times\R^{\tilde{d}_k} \to \R^{\tilde{d}_1}\times\ldots\times\R^{\tilde{d}_k} $ with $\tilde{A}_{i}\coloneqq(\tilde{A}_{i,1},\ldots,\tilde{A}_{i,k})$ and $\tilde{A}_{i,k}(\vc x) \coloneqq \zeta\Big(A_{i,k}\big(\xi(\vc x)\big)\Big)$ instead of the function $A_i$ defined in the proof of Theorem \ref{weak_strictpos} (note that the partial symmetry of $f$ implies $\xi\big(\tilde{A}_{i,k}(\vc x)\big)= A_{i,k}\big(\xi(\vc x)\big)$ for every $\vc x\in\R^{\tilde{d}_1}\times\ldots\times\R^{\tilde{d}_k}$). Now, we show the uniqueness of a strictly positive solution. Let $(\lambda,\vc x_1,\ldots,\vc x_k),(\mu,\vc u_1,\ldots,\vc u_k)$ be two strictly positive solutions to \eqref{singeig_pb}. By the uniqueness result of Theorem \ref{gen_sum_PF}, we know that $\lambda=\mu$ and $\xi(\vc x_1,\ldots,\vc x_k)=\xi(\vc u_1,\ldots,\vc u_k)$. $(\lambda,\vc x_1,\ldots,\vc x_k)=(\mu,\vc u_1,\ldots,\vc u_k),$ follows then from the injectivity of $\xi$ and thus there can be only one positive solution to Problem \eqref{singeig_pb}. If $f$ is irreducible, a similar argument shows that there can be only one nonnegative solution. Finally, suppose that $(\tilde{\vc x}_1,\ldots,\tilde{\vc x}_m)$ is the unique strictly positive $\ell^{p_1,\ldots,p_m}$-singular vector of $f$ and assume that there exists a strictly positive solution $(\vc x^*_1,\ldots,\vc x^*_k)$ to \eqref{singeig_pb}, we must have $\xi(\vc x^*_1,\ldots,\vc x^*_k)=(\tilde{\vc x}_1,\ldots,\tilde{\vc x}_m)$ and thus $\zeta(\tilde{\vc x}_1,\ldots,\tilde{\vc x}_m)=(\vc x^*_1,\ldots,\vc x^*_k)$.
\end{proof}
Note that the partial symmetry assumption of Proposition \ref{relation_prop} is crucial as shown by the following example.
\begin{ex}
Let $f\in\R^{2\times 2 \times 2}$ with $f_{1,2,2}=f_{2,1,2}=0$ and $f_{i,k,l}=1$ else. Let $1<p<\infty,k=2, q_1 = 1$ and $q_2=2$. Set $\vc x_1=\vc x_2\coloneqq \big(\frac{1}{2^{1/p}},\frac{1}{2^{1/p}}\big)$, then $\grad_1f(\vc x_1,\vc x_2,\vc x_2)=\grad_2f(\vc x_1,\vc x_2,\vc x_2)=\frac{3}{2^{3/p-1}}\psi_p\big(\frac{1}{2^{1/p}},\frac{1}{2^{1/p}}\big)$ and thus $\big(\frac{3}{2^{3/p-1}},\vc x_1,\vc x_2\big)$ is a solution of \eqref{singeig_pb}. However $\grad_3f(\vc x_1,\vc x_2,\vc x_2)=\frac{1}{2^{3/p-1}}\psi_p\big(2^{\frac{p+1}{p(p-1)}},2^{\frac{1}{p(p-1)}}\big)$, i.e. $(\vc x_1,\vc x_2,\vc x_2)$ is not a $\ell^{p,p,p}$-singular value of $f$.
\end{ex}
Problem \eqref{singeig_pb} models many of the eigenvalue problems for tensors. In particular if $k=1$, then we recover the $H$-eigenvalue problem for $p=m$ and the $Z$-eigenvalue problem for $p=2$. These problems were introduced in $2005$ by Qi \cite{Qi_eig}. Still for $k=1$, the more general problem for $1<p<\infty$ is known as $\ell^p$-eigenvalue problem and was introduced by Lim in \cite{Lim}. In \cite{Chang}, \cite{Lim} and \cite{Fried}, a Perron-Frobenius Theorem is proved for $\ell^p$-eigenvalues of tensors. The requirement on $p$ is $p\geq m$ and is equivalent to the condition of Theorem \ref{gen_sum_PF} when $p_1=\ldots=p_m$. The case $k=2$ and $\tilde{p}_1=\tilde{p}_2=2$, is known as $M$-eigenvalue problem and was introduced by Chang, Qi and Zhou in \cite{Chang_rect_eig}. The more general formulation for $k=2$ and $1<\tilde{p}_1,\tilde{p}_2<\infty$ is known as $\ell^{\tilde{p}_1,\tilde{p}_2}$-singular value problem for rectangular tensors and was introduced by Ling and Qi in \cite{Qi_rect_eig}. Ling and Qi also proved a Perron-Frobenius Theorem for $\ell^{\tilde{p}_1,\tilde{p}_2}$-singular value problems and the condition on $\tilde{p}_1,\tilde{p}_2$ is $\tilde{p}_1,\tilde{p}_2\geq m$. This condition is equivalent to ours whenever $q_1\notin\{1,m-1\}$, nevertheless, for the case $q_1=1$, we only require $m-1\leq (\tilde{p}_1-1)(\tilde{p}_2-m-1)$ and for the case $q_1=m-1$, our condition becomes $m-1\leq (\tilde{p}_2-1)(\tilde{p}_1-m-1)$.

\section{Computation of the tensor norm and singular vectors of a nonnegative tensor}\label{conv_sec}
We derive now an algorithm which takes benefit of the theory developed above. More precisely, motivated by the properties of $\mbd_i, \Mbd_i$ in Equation \eqref{gammadef}, we study the sequence $(\vc x^k)_{k \in \N} \subset \NS^{d-d_i}$ produced by the Higher-order Generalized Power Method ($\algo$). The $\algo$ is a tensor generalization of the algorithm proposed by Boyd in \cite{Boyd} for matrices (they coincide for $m=2$).
\begin{figure}
\begin{center}
\boxed{
\begin{tabular}{c}
Higher-order Generalized Power Method (\algo) \\ \hline
\begin{tabular}{l}
\underline{Input:} $f\in\R^{d_1\times\ldots\times d_m}$ and $p_1,\ldots,p_m$ satisfying assumptions of Theorem \ref{gen_sum_PF}, $\epsilon >0$.\\ 
\underline{Initialization:} $i\in [m]$ with $(m-1)p'_i \leq p_k$ for all $k\in [m]\setminus\{i\}$. (if $m=2$, choose $i\in [m]$ with \\ 
\color{white}{\underline{Initialization:}} \color{black} $p_i\leq p_k$ for $k\in [2]\setminus\{i\}$), $\vc x^0>0$ with $\norm{\vc x_l^0}_{p_l}=1$ for $l\in [m]$, $k=0$. \vspace{1mm}\\
\textbf{Do}\\
\qquad $\vc z^{k} = \big(s_{i,1}(\vc x^k),\ldots,s_{i,i-1}(\vc x^k),s_{i,i+1}(\vc x^k),\ldots,s_{i,m}(\vc x^k)\big)$ \vspace{1mm} \\
\qquad $\lambda^{k+1}_- = \displaystyle \prod_{l\in\msi} \min_{j_l \in [d_l]}\left(\frac{z^k_{l,j_l}}{x^k_{l,j_l}}\right)^{\frac{p_l-1}{p_i'(m-1)}}, \quad \lambda^{k+1}_+ = \displaystyle\prod_{l\in\msi} \max_{j_l \in [d_l]}\left(\frac{z^k_{l,j_l}}{x^k_{l,j_l}}\right)^{\frac{p_l-1}{p_i'(m-1)}}$\vspace{1mm} \\
\qquad $\vc x^{k+1} = \left(\frac{\vc z^k_1}{\norm{\vc z_1^k}_{p_1}},\ldots,\frac{\vc z^k_{i-1}}{\norm{\vc z_{i-1}^k}_{p_{i-1}}},\frac{\vc z^k_{i+1}}{\norm{\vc z_{i+1}^k}_{p_{i+1}}},\ldots,\frac{\vc z^k_m}{\norm{\vc z_m^k}_{p_m}}\right)\vspace{1mm} $\\
\qquad $k = k+1$ \vspace{1mm} \\
\textbf{Until} $(\lambda^{k}_+ - \lambda^{k}_-)<\epsilon$ \vspace{1mm} \\ 
\underline{Output:} Approximate of the maximal singular vector $\vc x\coloneqq \left(\vc x_1^{k},\ldots,\vc x_{i-1}^{k},\frac{\sigma_i(\vc x^k)}{\norm{\sigma_i(\vc x^k)}_{p_i}},\vc x_{i+1}^k,\ldots,\vc x_m\right)$ and  \\
\color{white}{\underline{Output:}} \color{black} approximate of the maximal singular value $\lambda\coloneqq f(\vc x)$. Moreover, $\left|\frac{\lambda^k_-+\lambda_+^k}{2}-\norm{f}_{p_1,\ldots,p_m}\right|< \epsilon$.
\end{tabular}
\end{tabular}}
\centering{\vspace{1mm}$\sigma_i$ is defined in Eq. \eqref{def_sig}, p. \pageref{def_sig} and $s_{i,k}$ is defined in Eq. \eqref{def_s}, p. \pageref{def_s}, to select $i\in [m]$ see Remark \ref{pi_prop}.}
\label{HGPM_alg}
\end{center}
\end{figure}
In order to prove the convergence of the sequence $(\vc x^k)_{k\in\N}$ produced by \algo{}, we show first that if the starting vector $\vc x^0$ is close enough to the singular vector $\vc x^* >0$ then we have linear convergence. Then, we show that for any starting point $\vc x^0\in\NS^{d-d_i}_{++}$ the sequence converges to $\vc x^*$. Let $G:\NS^{d-d_i}_{++} \to \NS^{d-d_i}_{++}$ be defined by
\begin{equation*}
G(\vc x)\coloneqq \left(\frac{s_{i,1}(\vc x)}{\norm{s_{i,1}(\vc x)}_{p_1}},\ldots,\frac{s_{i,m}(\vc x)}{\norm{s_{i,m}(\vc x)}_{p_m}}\right).
\end{equation*}
Note that if $f$ is weakly irreducible, by Corollary \ref{dual_s_pos}, $G$ is well defined and for $k\in\N$ we have $\vc x^{k+1}=G(\vc x^k)$ where $(\vc x^k)_{k\in\N}$ is the sequence produced by $\algo$. Furthermore, if $\vc x^*$ is a strictly positive critical point of $Q_i$, then $G(\vc x^*)=\vc x^*$. The next proposition gives some properties of the sequences $(\lambda_-^k)_{k\in\N},(\lambda_+^k)_{k\in\N}$ produced by \algo{} that motivate our choice for the stopping criterium.
\begin{prop}\label{mono}
Let $f\in\R^{d_1\times \ldots \times d_m},f \geq 0$ be a weakly irreducible tensor, $1 < p_1,\ldots,p_m < \infty$ such that there exists $i \in [m]$ with $(m-1)p_i' \leq p_k$ for every $k\in\msi$ and $\mbd_i,\Mbd_i$ as in Equation \eqref{gammadef}. Furthermore, consider the sequences $(\lambda_-^k)_{k\in\N}$ and $(\lambda_+^k)_{k\in\N}$ produced by \algo{}, then
\begin{equation*}
\lambda_-^{k}\ \leq \ \lambda_-^{k+1} \ \leq \ \norm{f}_{p_1,\ldots,p_m} \ \leq\  \lambda_+^{k+1}\ \leq \ \lambda_+^{k}  \qquad \forall k \in \N,
\end{equation*}
and when the Algorithm stops we have $\left|\frac{\lambda^k_-+\lambda_+^k}{2}-\norm{f}_{p_1,\ldots,p_m}\right|< \epsilon$.
\end{prop}
\begin{proof}
The proof is in two steps, first we prove that $\mbd_i(\vc z) \leq \mbd_i\big(G(\vc z)\big) \leq \norm{f}_{p_1,\ldots,p_m}^{p_i'(m-1)}\leq \Mbd_i\big(G(\vc z)\big) \leq  \Mbd_i(\vc z)$ for every $\vc z \in \NS_{++}^{d-d_i}$ and then we conclude the proof. Let $k\in\msi,\phi_k \coloneqq \min_{j_k\in[d_k]}\frac{s_{i,k,j_k}(\vc z)}{z_{k,j_k}}$ and $\varphi_k \coloneqq \max_{j_k\in[d_k]}\frac{s_{i,k,j_k}(\vc z)}{z_{k,j_k}}$, then we have 
$\mbd_i(\vc z)=\prod_{k\in\msi} \phi_k^{p_k-1}, \vc z_k \phi_k \leq s_{i,k}(\vc z),\Mbd_i(\vc z)=\prod_{k\in\msi} \varphi_k^{p_k-1} $ and $\vc z_k \varphi_k \geq s_{i,k}(\vc z).$ Note that 
\begin{equation*}
0<\frac{\phi_k}{\norm{s_{i,k}(\vc z)}_{p_k}} =\frac{\norm{\vc z_k\phi_k}_{p_k}}{\norm{s_{i,k}(\vc z)}_{p_k}} \leq \frac{\norm{s_{i,k}(\vc z)}_{p_k}}{\norm{s_{i,k}(\vc z)}_{p_k}}=1 \leq \frac{\norm{\vc z_k\varphi_k}_{p_k}}{\norm{s_{i,k}(\vc z)}_{p_k}} = \frac{\varphi_k}{\norm{s_{i,k}(\vc z)}_{p_k}} <\infty,
\end{equation*}
follows from $\vc z\in\NS^{d-d_i}$.
We only prove the inequality for $\mbd_i$ since the other inequality can be proved in the same way. By Lemmas \ref{order_preserving} and \ref{homos}, for any $l \in\msi$ and any $j_l \in [d_l]$ we have
\begin{equation*}
\psi_{p_l}\Big(s_{i,l,j_l}\big(G(\vc z)\big)\Big)\geq \left(\frac{\phi_l}{\norm{s_{i,l}(\vc z)}_{p_l}}\right)^{-1}\left(\prod_{k\in\msi} \frac{\phi_k}{\norm{s_{i,k}(\vc z)}_{p_k}}\right)^{p_i'}\psi_{p_l}\big(s_{i,l,j_l}(\vc z)\big).
\end{equation*}
With $\frac{\phi_k}{\norm{s_{i,k}(\vc z)}_{p_k}}\leq 1$ and $1<p_i'(m-1) \leq p_k$ for every $k \in\msi$, we get
\begin{equation*}
\prod_{k\in\msi}\left( \frac{\phi_k}{\norm{s_{i,k}(\vc z)}_{p_k}}\right)^{p_i'(m-1)-1}  \geq \prod_{k\in\msi}\left(\frac{\phi_k}{\norm{s_{i,k}(\vc z)}_{p_k}}\right)^{p_k-1}=\mbd_i(\vc z)\prod_{k\in\msi}\frac{1}{\norm{s_{i,k}(\vc z)}_{p_k}^{p_k-1}},
\end{equation*}
Combining these facts shows that for every $j_1\in [d_1], \ldots, j_m \in [d_m]$ holds
\begin{equation*}
\mbd_i(\vc z) \leq \prod_{l\in\msi}\dfrac{\ \ \psi_{p_l}\Big(s_{i,l,j_l}\big(G(\vc z)\big)\Big)\ \ }{\dfrac{\psi_{p_l}\big(s_{i,l,j_l}(\vc z)\big)}{\norm{s_{i,l}(\vc z)}_{p_l}^{p_l-1}}}= \prod_{l\in\msi}\ \frac{\big(s_{i,l,j_l}(\vc z^s)\big)^{p_l-1}}{\left(\dfrac{s_{i,l,j_l}(\vc z)}{\norm{s_{i,l}(\vc z)}_{p_l}}\right)^{p_l-1}}=\prod_{l\in\msi}\ \left(\frac{s_{i,l,j_l}\big(G(\vc z)\big)}{\big(G(\vc z)\big)_{l,j_l}}\right)^{p_l-1}.
\end{equation*}
Take the minimum over $j_1\in [d_1],\ldots,j_m \in [d_m]$ to get $
\mbd_i(\vc z) \leq \mbd_i\big(G(\vc z)\big).$ By Theorem \ref{gen_PF}, we know that $\mbd_i\big(G(\vc z)\big) \leq \norm{f}_{p_1,\ldots,p_m}^{p_i'(m-1)}\leq \Mbd_i\big(G(\vc z)\big)$. This concludes the first step of our proof. Now, if $(\vc x^k)_{k\in\N}$ is the sequence produced by \algo{}, then, by Corollary \ref{dual_s_pos}, $(\vc x^k)_{k\in\N}\subset \NS^{d-d_i}_{++}$ since $\vc x^0 \in \NS^{d-d_i}_{++}$ by assumption. Moreover, note that $\vc x^{k+1}=G(\vc x^k),\lambda^{k+1}_-=\big(\mbd_i(\vc x^k)\big)^{\frac{1}{p_i'(m-1)}}$ and $\lambda^{k+1}_+=\big(\Mbd_i(\vc x^k)\big)^{\frac{1}{p_i'(m-1)}}$ for every $k\in\N$. It follows that $\lambda_-^{k}\leq \lambda_-^{k+1} \leq \norm{f}_{p_1,\ldots,p_m} \leq \lambda_+^{k+1}\leq \lambda_+^{k}$ for every $k\in \N$. Finally, if $(\lambda_+^k-\lambda_-^k)<\epsilon$, subtracting $\frac{\lambda_+^k-\lambda_-^k}{2}$ from the inequality $\lambda_-^{k}\leq \norm{f}_{p_1,\ldots,p_m}\leq \lambda_+^{k}$, shows
\begin{equation*}-\epsilon < -\frac{\lambda_+^k-\lambda_-^k}{2} \leq \norm{f}_{p_1,\ldots,p_m}\leq \frac{\lambda_+^k-\lambda_-^k}{2}<\epsilon.\qedhere
\end{equation*}
\end{proof}
Now, we prove the convergence of the sequences produced by \algo{}.
\begin{lem}\label{prop_DF}
Let $f\geq 0$ be a weakly irreducible tensor and $1 < p_1,\ldots,p_m < \infty$ such that there exists $i \in [m]$ with $(m-1)p_i' \leq p_k$ for all $k\in\msi$. Let $\vc x^*\in \NS^{d-d_i}_{++}$ be a critical point of $Q_i$ and $\lambda \coloneqq Q_i(\vc x^*)$. Furthermore, consider the function $F:\sR^{d-d_i} \to \sR^{d-d_i}$ defined by $F(\vc x)\coloneqq\big(F_1(\vc x),\ldots,F_m(\vc x)\big)$ where for each $k\in\msi$, $
F_k(\vc x) \coloneqq \lambda^{1-p_i'(p_{k}'-1)}\norm{\vc x_k}_{p_k}^{p_k'+(m-1)p_i'(1-p'_k)+\rho}s_{i,k}(\vc x)
$ and $\rho > 0$ is such that $p_l'+(m-1)p_i'(1-p'_l)+\rho>0$ for all $l\in [m]\setminus\{i\}$. Let us denote by $B$ the Jacobian matrix of $F$ at $\vc x^*$. Then the function $F$ is positively $(\rho+1)$-homogeneous, $B$ is primitive with Perron-root $\lambda_1 \coloneqq (\rho+1)\lambda$ and for any $\vc g, \vc h\in \sR^{d-d_i},k\in\msi$, we have 
\begin{equation*}
\ps{|\vc x^*_k|^{p_k-2}\hap \vc h_k}{(B\vc g)_k} = \ps{|\vc x^*_k|^{p_k-2}\hap \vc g_k}{(B\vc h)_k},
\end{equation*}
where $\hap $ denotes the Hadamard product and $|\vc y|^{q}\coloneqq \big(|y_1|^q,\ldots,|y_n|^q\big)$ for $\vc y\in \R^n, q \in \R$.
\end{lem}
\begin{proof}
Since $\vc x^* >0$, there exists some open neighborhood $U \subset \sR^{d-d_i}_{++}$ of $\vc x^*$ such that $s_{i,k}$ is smooth on $U$ for any $k\in\msi$. Existence of such a neighborhood is obvious since $f(\vc x)$ and $\grad_if(\vc x)$ are smooth on $\sR^d$. Furthermore, $f$ is weakly irreducible thus, by Proposition \ref{Characterizer_prop}, we have $\partial_{k,j_k}f(\vc x)> 0$ for any $\vc x> 0$ and $\psi_p(t)$ is smooth on $\R\setminus\{0\}$. With help of Lemma \ref{homos} it is straightforward to check that $F$ is positively $(\rho+1)$-homogeneous. Let $
\alpha_{k}(\vc x) \coloneqq \lambda^{1-p_i'(p_k'-1)}\norm{\vc x_k}_{p_k}^{p_k'+(m-1)p_i'(1-p'_k)+\rho},$ then for every $\vc x \in U$ and $l\in\msi,$ we have
$\grad_l\alpha_{k}(\vc x)= \lambda^{1-p_i'(p_k'-1)}\big(p_k'+(m-1)p_i'(1-p'_k)+\rho\big) \norm{\vc x_k}_{p_k}^{p_k'+(m-1)p_i'(1-p_k)+\rho-p_k}\psi_{p_k}(\vc x_k)$ if $l = k$ and $\grad_l\alpha_{k}(\vc x)= 0$ if $k \neq l$. In particular, $\vc x^* \in \NS^d_{++}$ thus $\grad_k \alpha_k(\vc x^*) >0$. Note that  for every $k,l\in\msi,j_l \in [d_l]$ and $j_k\in[d_k]$, we have $
\partial_{l,j_l} F_{k,j_k}(\vc x) =s_{i,k,j_k}(\vc x) \partial_{l,j_l} \alpha_k(\vc x)+ \alpha_k(\vc x)\partial_{l,j_l}s_{i,k,j_k}(\vc x)$. Moreover, using $(p'-2)=(2-p)(p'-1)$ for $1<p<\infty$, a tedious computation shows
\begin{align*}
\partial_{l,j_l}s_{i,k,j_k}(\vc x^*) =& (p_i'-1)(p_k'-1)\lambda^{p_i'(p_k'-1)-2}\big|x^*_{k,j_k}\big|^{2-p_k}\ps{\big|\vc x^*_i\big|^{2-p_i}\hap \grad_i \partial_{l,j_l}f(\vc x^*)}{\grad_i \partial_{k,j_k}f(\vc x^*)}\\ &+(p_k'-1)\lambda^{p_i'(p_k'-1)-1}\big|x^*_{k,j_k}\big|^{2-p_k}\partial_{l,j_l}\partial_{k,j_k}f(\vc x^*) \qquad \forall k,l\in [m], j_k\in[d_k],j_l\in[d_l].
\end{align*}
In particular, note that $\partial_{l,j_l}s_{i,k,j_k}(\vc x^*)\geq 0$ follows from $f \geq 0$ and $\vc x^* >0$, thus $B \geq 0$ because $s_{i,k,j_k}(\vc x^*)\geq 0$ by Lemma \ref{order_preserving}, $\alpha_k(\vc x^*)\geq 0$ and $\partial_{l,j_l}\alpha_k(\vc x^*)\geq 0$.
From Equation \eqref{dual_eigensys} and $\vc x^* \in \NS^d$, we know that $s_{i,k}(\vc x^*)=\lambda^{p_i'(p_k'-1)}\vc x^*_k > 0$ and $\partial_{k,j_k}\alpha_k(\vc x^*)=\big(p_k'+(m-1)p_i'(1-p'_k)+\rho\big)\lambda^{1-p_i'(p_k'-1)}\psi_{p_k}(x_{k,j_k}^*)>0$, thus 
\begin{align*}
B_{(k,j_k),(k,l_k)}&=s_{i,k,j_k}(\vc x^*) \partial_{k,l_k} \alpha_k(\vc x^*)+\alpha_k(\vc x^*)\partial_{k,l_k}s_{i,k,j_k}(\vc x^*) \\ &\geq s_{i,k,j_k}(\vc x^*) \partial_{k,l_k} \alpha_k(\vc x^*) = \lambda\ \big(p_k'+(m-1)p_i'(1-p'_k)+\rho\big)\big(x^*_{k,j_k}\big)^{p_k} >0
\end{align*}
for all $k \in \msi$ and $j_k,l_k \in [d_k].$ This shows that the matrix $B$ has strictly positive blocks of size $d_k\times d_k$ on its main diagonal. In order to prove that $B$ is an irreducible matrix, we show that for every $k,l \in [m]\setminus\{i\}$ with $k\neq l$ there exists $j_k \in [d_k]$ and $j_l \in [d_l]$ such that $\max\big\{B_{(k,j_k),(l,j_l)},B_{(l,j_l),(k,j_k)}\big\} >0$. This would imply that the graph associated to the adjacency matrix $B$ is connected. Fix $k,l\in\msi$ with $k\neq l$ and suppose by contradiction that for every $j_k \in [d_k]$ and every $j_l \in [d_l],$ we have $
0=B_{(k,j_k),(l,j_l)}=B_{(l,j_l),(k,j_k)}$. It follows that $\partial_{l,j_l}s_{i,k,j_k}(\vc x^*)=0$ and $\partial_{l,j_l}\partial_{k,j_k}f(\vc x^*)=0$ for every $j_k\in [d_k],j_l\in [d_l]$. So, $0 = \partial_{l,j_l}f(\vc x^*)=\ps{\vc x_k^*}{\grad_k\partial_{l,j_l}f(\vc x^*)}$ for every $j_l\in [d_l]$, thus $0=\ps{\vc x_l}{\grad_lf(\vc x^*)}=f(\vc x^*) = \norm{f}_{p_1,\ldots,p_m}$, a contradiction. Hence the graph associated to the nonnegative adjacency matrix $B$ is connected and $B$ has strictly positive main diagonal entries, it follows from Lemma 8.5.5 and Theorem 8.5.2 in \cite{Horn} that $B$ is primitive. Now, using $s_{i,k}(\vc x^*) = \lambda^{p_i'(p_k'-1)} \vc x_k^*$ it can be observed that $F(\vc x^*) = \lambda\vc x^*$. Approximating $\alpha \mapsto F(\alpha \vc x^*)$ linearly at $1$ shows $
\alpha^{1+\rho} \lambda \vc x^* =\alpha^{1+\rho} F( \vc x^*) = F( \alpha \vc x^*) = F(\vc x^*) + (\alpha-1)B\vc x^* +o(\alpha -1)$ for every $\alpha > 1$.
Subtract $\lambda \vc x^* + (\alpha-1)B\vc x^* $ on both sides and take the limit $\alpha\to 1$ to get
\begin{equation*}
(1+\rho)\lambda \vc x^*-B\vc x^*= \lim_{\alpha \downarrow 1}\frac{(\alpha^{1+\rho}-1)\lambda \vc x^*-(\alpha-1)B\vc x^*}{(\alpha -1)} =0,
\end{equation*}
i.e. $(1+\rho)\lambda>0$ is the Perron-root of $B$ associated to the strictly positive eigenvector $\vc x^*$. Now, we prove that $
\ps{|\vc x^*_k|^{p_k-2}\hap \vc h_k}{(B\vc g)_k} = \ps{|\vc x^*_k|^{p_k-2}\hap \vc g_k}{(B\vc h)_k}.$ Since $F$ is differentiable on $U$, for every $\vc h\in \sR^{d-d_i}$ we have $
(B\vc h)_k = \big(\delta F(\vc x^*;\vc h)\big)_k = \delta F_k(\vc x^*;\vc h)$ where 
$
\delta F(\vc x; \vc h) \coloneqq \lim_{\epsilon \to 0}\epsilon^{-1}\big(F(\vc x + \epsilon \vc h)-F(\vc x)\big)
$
is the directional derivative of $F$ at $\vc x$ in the direction $\vc h$. The multiplication rule for directional derivative, $\vc x^*\in\NS^{d-d_i}$ and $s_{i,k}(\vc x^*)=\lambda^{p_i'(p_k'-1)}\vc x^*$, show
\begin{equation*}
\delta F_k(\vc x^*;\vc h) = \alpha_k(\vc x^*)\delta s_{i,k}(\vc x^*;\vc h) + s_{i,k}(\vc x^*)\delta \alpha_k(\vc x^*;\vc h) = \lambda^{1-p_i'(p_k'-1)}\delta s_{i,k}(\vc x^*;\vc h) + \lambda^{p_i'(p_k'-1)} \vc x_k^*\ps{\grad \alpha_k(\vc x^*)}{\vc h}.
\end{equation*}
Let $C \coloneqq \lambda^{1-p_i'(p_k'-1)}\big(p_k'+(m-1)p_i'(1-p_k)+\rho\big)$, then
\begin{equation*}
\ps{|\vc x^*_k|^{p_k-2}\hap \vc g_k}{\vc x_k^*\ps{\grad \alpha_k(\vc x^*)}{\vc h}}= C \ps{\psi_{p_k}(\vc x^*_k)}{\vc h_k}\ps{\psi_{p_k}(\vc x_k^*)}{\vc g_k} =\ps{|\vc x^*_k|^{p_k-2}\hap \vc h_k}{\vc x_k^*\ps{\grad \alpha_k(\vc x^*)}{\vc g}}.
\end{equation*}
In order to conclude the proof, we show $
\ps{|\vc x^*_k|^{p_k-2}\hap \vc h_k}{\delta s_{i,k}(\vc x^*;\vc g)} = \ps{|\vc x^*_k|^{p_k-2}\hap \vc g_k}{\delta s_{i,k}(\vc x^*;\vc h)}.$ Another tedious computation shows $
 \delta s_{i,k}(\vc x; \vc h) = (p_k'-1)\big|s_{i,k}(\vc x)\big|^{2-p_k} \hap L_{i,k}(\vc x; \vc h)$ for all $\vc x \in U$ where
 \begin{equation*}
 \begin{array}{l}
L_{i,k}(\vc x,\vc h) \coloneqq \Sum_{l\in[m]\setminus\{k,i\}}\grad_{k} f\bigl \vc x_1,\ldots ,\vc x_{l-1}, \vc h_l, \vc x_{l+1} ,\ldots ,\vc x_{i-1}, \psi_{p_i'}\big(\grad_{i}f(\vc x)\big),\vc x_{i+1},\ldots, \vc x_m \bigr \\
 \ + (p_i'-1)\Sum_{s\in[m]\setminus\{k\}}\grad_{k} f\bigl\vc x_1,\ldots ,\vc x_{i-1}, \big|\grad_{i}f(\vc x)\big|^{p_i'-2}\hap\grad_{i} f(\vc x_1, \ldots, \vc x_{s-1}, \vc h_s, \vc x_{s+1},\ldots, \vc x_{m}),\vc x_{i+1},\ldots, \vc x_m\bigr.
 \end{array}
 \end{equation*}
 Note that for any $k,l\in\msi$ with $k \neq l$ holds
  \begin{eqnarray*}
  \ps{\grad_{k} f(\vc x_1,\ldots ,\vc x_{l-1}, \vc h_l, \vc x_{l+1} ,\ldots, \vc x_m)}{\vc g_k} &=& f(\vc x_1, \ldots,\vc x_{l-1},\vc h_l, \vc x_{l+1},\ldots,\vc x_{k-1}, \vc g_k, \vc x_{k+1}, \ldots, \vc x_m)\\
  &=& \ps{\grad_{l} f(\vc x_1,\ldots ,\vc x_{k-1}, \vc g_k, \vc x_{k+1} ,\ldots, \vc x_m)}{\vc h_l},
  \end{eqnarray*}
  furthermore for $s \in\msi$, we have
  \begin{equation*}
  \begin{array}{l}
  \ps{\grad_{k} f\bigl\vc x_1,\ldots ,\vc x_{i-1}, \big|\grad_{i}f(\vc x)\big|^{p_i'-2}\hap\grad_{i} f(\vc x_1, \ldots, \vc x_{s-1}, \vc h_s, \vc x_{s+1},\ldots, \vc x_{m}),\vc x_{i+1},\ldots, \vc x_m\bigr}{\vc g_k}\\
   \ = f\bigl\vc x_1,\ldots ,\vc x_{i-1}, \big|\grad_{i}f(\vc x)\big|^{p_i'-2}\hap\grad_{i} f(\vc x_1, \ldots, \vc x_{s-1}, \vc h_s, \vc x_{s+1},\ldots, \vc x_{m}),\vc x_{i+1},\ldots,\vc x_{k-1},\vc g_k, \vc x_{k+1}, \ldots, \vc x_m\bigr \\
  \ = \ps{\grad_{i}f(\vc x_1, \ldots, \vc x_{k-1}, \vc g_{k}, \vc x_{k+1},\ldots,\vc x_m)}{\big|\grad_{i}f(\vc x)\big|^{p_i'-2}\hap\grad_{i}f(\vc x_1, \ldots, \vc x_{s-1}, \vc h_{s}, \vc x_{s+1},\ldots,\vc x_m)} \\
  \ = \ps{\big|\grad_{i}f(\vc x)\big|^{p_i'-2}\hap\grad_{i}f(\vc x_1, \ldots, \vc x_{k-1}, \vc g_{k}, \vc x_{k+1},\ldots,\vc x_m)}{\grad_{i}f(\vc x_1, \ldots, \vc x_{s-1}, \vc h_{s}, \vc x_{s+1},\ldots,\vc x_m)} \\
  \ = \ps{\grad_{s} f\bigl \vc x_1,\ldots ,\vc x_{i-1}, \big|\grad_{i}f(\vc x)\big|^{p_i'-2}\hap\grad_{\vc x_i} f(\vc x_1, \ldots, \vc x_{k-1}, \vc g_k, \vc x_{k+1},\ldots, \vc x_{m}),\vc x_{i+1},\ldots, \vc x_m\bigr}{\vc h_s}.
  \end{array}
  \end{equation*}
  These relations imply $\ps{L_{i,k}(\vc x; \vc h)}{\vc g_k} = \ps{L_{i,k}(\vc x; \vc g)}{\vc h_k}$ for all $\vc x \in U$. Thus, for every $k\in\msi$, we have 
  \begin{equation*}
  \ps{|\vc x^*_k|^{p_k-2}\hap \vc g_k}{\delta s_{i,k}(\vc x^*;\vc h)}  =  (p_k'-1)\lambda^{p_i'(p_k'-2)}\ps{L_{i,k}(\vc x; \vc h)}{\vc g_k} 
  = \ps{|\vc x^*_k|^{p_k-2}\hap \vc h_k}{\delta s_{i,k}(\vc x^*;\vc g)}.\qedhere
  \end{equation*}
  \end{proof}
\begin{prop} \label{specrad}
Let $f\geq 0$ a weakly irreducible tensor, $1 < p_1,\ldots,p_m < \infty$ such that there is $i \in [m]$ with $(m-1)p_i' \leq p_k$ for every $k\in\msi$, $\vc x^*$ the unique strictly positive critical point of $Q_i$ in $\NS^{d-d_i}$. Moreover, let $DG(\vc x^*)$ be the Jacobian matrix of $G$ at $\vc x^*$, then the spectral radius of $DG(\vc x^*)$ is strictly smaller than $1$, i.e. $\rho\big(DG(\vc x^*)\big)<1$.
\end{prop}
\begin{proof}
 Let $B$ as in Lemma \ref{prop_DF}, then $B$ is primitive and has Perron root $\lambda_1 = (\rho+1)\norm{f}_{p_1,\ldots,p_m}$. Let $d \coloneqq d_1 + \ldots + d_m$ and $\lambda_1,\ldots, \lambda_{d-d_i}$ be the eigenvalues of $B$. We may order them so that $
\lambda_1 > |\lambda_2 |\geq |\lambda_3 |\geq  \ldots \geq |\lambda_{d-d_i}|.$
We claim that $|\lambda_2|$ is the spectral radius $\rho(M)$ of the matrix $M\coloneqq \lambda_1 DG(\vc x^*)$. Let $U\subset\sR^{d-d_i}_{++}$ an open neighborhood of $\vc x^*$ such that $F$ and $G$ are smooth in $U$. Observe that Corollary \ref{dual_s_pos} implies $F_k(\vc x)>0$ for every $\vc x \in U$ and each $k\in\msi$. Moreover, $G(\vc x) = \left(\frac{F_1(\vc x)}{\norm{F_1(\vc x)}_{p_1}},\ldots,\frac{F_m(\vc x)}{\norm{F_m(\vc x)}_{p_m}}\right)$ for every $\vc x \in U$ and thus, in particular, $G(\vc x^*)=\vc x^*$. A straightforward computation shows that 
$(M \vc h)_l = (B \vc h)_l - \vc x_l^*\ps{\psi_{p_l}(\vc x_l^*)}{(B\vc h)_l}$ for every $l\in\msi$ and every $\vc h\in \sR^{d-d_i}$. It follows that $
(M \vc x^*)_l = (B \vc x^*)_l -\vc x_l^*\ps{\psi_{p_l}(\vc x_l^*)}{(B\vc x^*)_l}= \lambda_1 \vc x^*_l -\lambda_1 \vc x_l^*\ps{\psi_{p_l}(\vc x_l^*)}{\vc x^*_l}= 0,$ i.e. $\vc x^*$ is an eigenvector of $M$ associated to the eigenvalue $0$. First, suppose that the eigenvalues of $B$ are all distinct and different from $0$. Now, from Proposition \ref{prop_DF}, we know that for every $\vc h, \vc g$ holds $\ps{|\vc x^*_k|^{p_k-2}\hap \vc h_k}{(B\vc g)_k} = \ps{|\vc x^*_k|^{p_k-2}\hap \vc g_k}{(B\vc h)_k}.$ For $s \in [d], s \geq 2$ let us denote by $\vc u^s$ the eigenvector of $B$ such that $B\vc u^s = \lambda_s \vc u^s$, then 
\begin{equation*}\lambda_s\ps{\psi_{p_k}(\vc x^*_k)}{\vc u^s_k} = \ps{|\vc x^*_k|^{p_k-2}\hap \vc x_k^*}{(B\vc u^s)_k}= \ps{|\vc x^*_k|^{p_k-2}\hap \vc u_k^s}{(B\vc x^*)_k} =\lambda_1 \ps{\psi_{p_k}(\vc x_k^*)}{\vc u_k^s}.\end{equation*}
From $\lambda_s \neq 0$ and $\lambda_1 \neq \lambda_s$ follows $\ps{\psi_{p_k}(\vc x_k^*)}{\vc u_k^s}=0$. This shows that for every $k \in\msi$ we have
$(M\vc u^s)_{k} = \lambda_s\vc u^s_{k} - \vc x^* \lambda_s\ps{\psi_{p_k}(\vc x_k^*)}{\vc u_k^s} = \lambda_s \vc u^s_{k},
$ i.e. $\lambda_s$ is an eigenvalue of $M$ associated to the eigenvector $\vc u^s$. Thus, the eigenvalues of $M$ are exactly $0,\lambda_2,\ldots,\lambda_{d-d_i}$ and we have $\rho(M) = |\lambda_2|$. It follows that $\rho\big(DG(\vc x^*)\big)=\frac{|\lambda_2|}{\lambda_1}<1$. The case when $B$ does not satisfy the above assumption is treated as in Corollary 5.2, \cite{Fried}.
\end{proof}
\begin{cor}\label{Neighborhood}
Let $f,p_1,\ldots,p_m,\vc x^*$ and $G$ as in Proposition \ref{specrad}. Then, there exists a norm $\norm{\cdot}_G$ on $\sR^{d-d_i}$, $r_0>0$ and $0< \nu <1$ such that for every $\vc y^0\in \sR^{d-d_i}$ with $\norm{\vc y^0-\vc x^*}_G<r_0$ we have 
\begin{equation}\label{lin_conv}
\norm{\vc y^{k+1}-\vc x^*}_G \leq \nu \norm{\vc y^{k}-\vc x^*}_G \qquad \forall k\in\N,
\end{equation}
where $\vc y^{k+1}\coloneqq G(\vc y^k)$ for every $k\in\N$.
\end{cor}
\begin{proof}
Let $U$ be an open neighborhood of $\vc x^*$ such that $G$ is differentiable on $U$ and $\vc u>0$ for every $\vc u\in U$. By Proposition \ref{specrad} we know that $\rho\big(DG(\vc x^*)\big) <1$, so let $L,\nu>0$ such that $\rho\big(DG(\vc x^*)\big) <L<\nu<1$. It is a classical result (see Lemma 5.6.10, \cite{Horn}) that there exists a norm $\norm{\cdot}_G$ on $\sR^{d-d_i}$ such that for every $\vc v\in\sR^{d-d_i}$ holds $\norm{DG(\vc x^*)\vc v}_G \leq L\norm{\vc v}_G$. Since $G$ is differentiable at $\vc x^*$, we have $G(\vc u) = G(\vc x^*) + DG(\vc x^*)(\vc u-\vc x^*) + R(\vc u,\vc x^*),$ with $\lim_{\vc u \to \vc x^*}\frac{\norm{R(\vc u,\vc x^*)}_G}{\norm{\vc u,\vc x^*}_G}=0$. Let $\epsilon > 0$ be such that $\epsilon \leq \nu -L$, then there is $r_0 > 0$ such that for every $\vc u\in \NS^{d-d_i}_{++}$ with $\norm{\vc u-\vc x^*}_G<r_0$ we have
$\norm{R\big(\vc u,\vc x^*\big)}_G \leq \epsilon\norm{\vc u-\vc x^*}_G$. By decreasing $r_0$ if necessary, we may suppose that  $\norm{\vc u-\vc x^*}_G<r_0$ implies $\vc u\in U$. It follows that for every $\vc y^0\in \NS^{d-d_i}$ with $\norm{\vc y^0-\vc x^*}_G<r_0$, it holds
\begin{align*}
\norm{\vc y^{1}-\vc x^*}_G &= \norm{G(\vc y^0)-G(\vc x^*)}_G  = \norm{DG(\vc x^*)(\vc y^0-\vc x^*) +R(\vc u,\vc x^*)}_G \\
& \leq  \norm{DG(\vc x^*)(\vc y^0-\vc x^*)}_G+ \norm{R(\vc u,\vc x^*)}_G \leq (L+\epsilon)\norm{\vc y^0-\vc x^*}_G\leq \nu\norm{\vc y^0-\vc x^*}_G.
\end{align*}
Since $\nu <1,$ we have $
\norm{\vc y^{1}-\vc x^*}_G \leq \nu\norm{\vc y^{0}-\vc x^*}_G <\norm{\vc y^0-\vc x^*}_G<r_0$. Thus we may
apply recursively the above argument to get $\norm{\vc y^{k+1}-\vc x^*}_G \leq \nu \norm{\vc y^{k}-\vc x^*}_G$ for every $k\in\N$.
\end{proof}
It follows directly that if $\vc x^0$ is close enough to the strictly positive singular vector $\vc x^*$ of $f$ (i.e. $\norm{\vc x^0-\vc x^*}_G < r_0$ with the notations of Corollary \ref{Neighborhood}), the sequence $(\vc x^k)_{k\in\N}$ produced by \algo{} has $\vc x^*$ as its limit and the convergence rate is linear. However, this result is of little practical use if we don't know $\vc x^*$. This is why we show now that for every $\vc x^0\in\NS^{d-d_i}_{++}$, the sequence $(\vc x^k)_{k\in\N}$ converges to $\vc x^*$. 
In order to do it, we use a Lemma proved by R. Nussbaum \cite{Nussb}.
\begin{lem}[Lemma 2.3, \cite{Nussb}]\label{golden_lemma}
Let $(S,\mu_S)$ be a metric space with metric $\mu_S$ and suppose that $S_0$ is a  connected subset of $S$ and $\mu$ is a metric on $S_0$ which gives the same topology on $S_0$ as that inherited from $S$. Let $T: S_0\to S_0$ be a map such that $\mu\big(T(\vc x),T(\vc y)\big)\leq \mu(\vc x,\vc y)$ for every $\vc x,\vc y\in S_0$. For $n\in\N$ and $\vc x \in S_0$, let $T^{n}(\vc x)\coloneqq T^{n-1}\big(T(\vc x)\big)$. Assume that there exists $\vc x^*\in S_0$ and an open neighborhood $U$ of $\vc x^*$ such that $U\cap S_0\neq\emptyset$ and $\lim_{n\to\infty}\mu_S\big(T^{n}(\vc x),\vc x^*\big)=0$ for every $\vc x\in U\cap S_0.$ Finally, if there exists a continuous map $\varphi:\{t\in\R\mid t\geq 0\} \to \{t\in\R\mid t \geq 0\}$ with $\varphi(0)=0$ such that $\mu_S(\vc x,\vc y)\leq \varphi\big(\mu(\vc x, \vc y)\big)$ for all $\vc x,\vc y\in U\cap S_0$. 
Then $\lim_{n\to \infty}\mu_S\big(T^{n}(\vc x),\vc x^*\big)=0$ for every $\vc x\in S_0$.
\end{lem}
We will apply Lemma \ref{golden_lemma} with $S_0 = \NS_{++}^{d-d_i},S=\sR^{d-d_i}, \mu_S(\vc x,\vc y)\coloneqq \norm{\vc x-\vc y}_G$, $T = G$ and $U = \big\{\vc x \in \sR^{d-d_i}\mid \norm{\vc x-\vc x^*}_G < r_0\big\}$ where $\norm{\cdot}_G$ and $r_0>0$ are such that Equation \eqref{lin_conv} is satisfied. Now, we build the metric $\mu$ on $S_0$ such that $\mu\big(G(\vc x),G(\vc y)\big)\leq \mu(\vc x,\vc y)$ for $\vc x,\vc y\in\NS_{++}^{d-d_i}$ and prove that the topology on $\big(\NS_{++}^{d-d_i},\mu\big)$ is the same as that inherited from $\big(\sR^{d-d_i},\mu_S\big)$.
\begin{prop}\label{metric_space}
Let $1<p_1,\ldots,p_m<\infty,i\in [m]$ and $\mu \colon \NS_{++}^{d-d_i} \times \NS_{++}^{d-d_i} \to \R$ defined by 
\begin{equation*}
\mu(\vc x,\vc y)\coloneqq \ln\left(\prod_{l\in\msi}\dfrac{\displaystyle\max_{j_l\in [d_l]}\left(\frac{x_{l,j_l}}{y_{l,j_l}}\right)^{p_l-1}}{\displaystyle\min_{j_l\in [d_l]}\left(\frac{x_{l,j_l}}{y_{l,j_l}}\right)^{p_l-1}}\right),
\end{equation*}
then $\big(\NS_{++}^{d-d_i},\mu\big)$ is a metric space.
\end{prop}
\begin{proof}
For $l\in [m],$ let $\NS^{d_l}_{++}\coloneqq\big\{\vc x_l\in\R^{d_l}_{++}\mid \norm{\vc x_l}_{p_l}=1\big\}$ and $\mu_l,\mu_l^-,\mu_l^+\colon \NS^{d_l}_{++}\times\NS^{d_l}_{++}\to \R$ with
\begin{equation}\label{def_mu_l}
\mu_l^-(\vc x_l, \vc y_l) \coloneqq \min_{j_l\in[d_l]} \left(\frac{x_{l,j_l}}{y_{l,j_l}}\right)^{p_l-1},\mu_l^+(\vc x_l, \vc y_l) \coloneqq \max_{j_l\in[d_l]} \left(\frac{x_{l,j_l}}{y_{l,j_l}}\right)^{p_l-1} \text{ and }\ \mu_l(\vc x_l,\vc y_l) \coloneqq \ln\left(\frac{\mu_l^+(\vc x_l,\vc y_l)}{\mu_l^{-}(\vc x_l,\vc y_l)}\right).
\end{equation}
It follows from Theorem 1.2 in \cite{Nussb} that $(\NS^{d_l}_{++},\mu_l)$ is a complete metric space. Note that $\NS^{d-d_i}_{++}= \NS^{d_1}_{++}\times\ldots\times\NS^{d_{i-1}}_{++}\times\NS^{d_{i+1}}_{++}\times\ldots\times\NS^{d_m}_{++}$ and $\mu(\vc x, \vc y)=\sum_{l\in\msi} \mu_l(\vc x_l,\vc y_l)$, i.e. $\big(\NS_{++}^{d-d_i},\mu\big)$ is the product metric space of  $\big(\NS_{++}^{d_1},\mu_1\big),\ldots,\big(\NS_{++}^{d_{i-1}},\mu_{i-1}\big),\big(\NS_{++}^{d_{i+1}},\mu_{i+1}\big),\ldots,\big(\NS_{++}^{d_m},\mu_m\big)$ and thus is a metric space itself.
\end{proof}
\begin{prop}\label{loc_equiv}
Let $f\geq 0$ be a weakly irreducible and $1<p_1,\ldots,p_m<\infty$ are such that there is $i\in [m]$ with $(m-1)p_i'\leq p_k$ for all $k\in [m]\setminus\{i\}$ and, if $m=2$, choose $i$ such that $p_i \leq p_k$ for $k\in [2]\setminus\{i\}$. Let $\norm{\cdot}_G$ be a norm such that Equation \eqref{lin_conv} is satisfied and $\mu$ the metric of Proposition \ref{metric_space}. Then, for every $\vc u \in \NS^{d-d_i}_{++}$, there exist $r>0$ and $c,C>0$ such that
\begin{equation*}
c\norm{\vc u-\vc x}_G\leq \mu(\vc u,\vc x)\leq C\norm{\vc u-\vc x}_G \qquad \forall \vc x \in \big\{\vc x \in \sR^{d-d_i}\mid \norm{\vc x-\vc u}_G < r\big\}\cap \NS^{d-d_i}_{++}.
\end{equation*}
\end{prop}
Note that the assumption $p_i \leq p_k$ for $k\in [m]\setminus\{i\}$ when $m=2$ is not restrictive by Remark \ref{pi_prop}.
\begin{proof}
Let $l\in\msi$ and $\vc v_l \in \NS^{d_l}_{++}$, note that since $\R_+\to \R_+\colon t\mapsto t^{p_l-1}$ is an increasing function ($p_l>1$ by assumption), for every $\vc  x_l,\vc y_l \in \R^{d_l}_{++}$, we have 
\begin{equation*}
\mu_l(\vc x_l,\vc y_l)= (p_l-1)\ln\left(\frac{\displaystyle\max_{j_l\in[d_l]} \frac{x_{l,j_l}}{y_{l,j_l}}}{\displaystyle\min_{j_l\in[d_l]}\frac{x_{l,j_l}}{y_{l,j_l}}}\right).
\end{equation*}
There exists $\xi_l >0$ such that $B_{\xi_l}^l(\vc v_l)\subset \R^{d_l}_{++}$. Equation (1.21) in \cite{Nussb} reads
\begin{equation*}
\frac{1}{p_l-1}\mu(\vc v_l, \vc x_l)\eqqcolon\tilde{\mu_l}(\vc v_l,\vc x_l) \leq \ln\left(\frac{\xi_l+\norm{\vc v_l-\vc x_l}_{p_l}}{\xi_l-\norm{\vc v_l - \vc x_l}_{p_l}}\right) = \ln\left(\frac{1+\frac{\norm{\vc v_l-\vc x_l}_{p_l}}{\xi_l}}{1-\frac{\norm{\vc v_l - \vc x_l}_{p_l}}{\xi_l}}\right) \qquad \forall \vc x_l \in B_{\xi_l}^l(\vc v_l).
\end{equation*}
So, let $t(\vc x_l) \coloneqq\frac{\norm{\vc v_l - \vc x_l}_{p_l}}{\xi_l}$ and $0<\epsilon_l< \xi_l$, then $\overline{B_{\epsilon_l}^l(\vc v_l)}\subset B_{\xi_l}^l(\vc v_l)$ and we have
$\tilde{\mu}_l(\vc v_l,\vc x_l) \leq \ln\left(\frac{1+t(\vc x_l)}{1-t(\vc x_l)}\right)$ for all $\vc x_l\in \overline{B_{\epsilon_l}^l(\vc v_l)}.$
Now, the function $h\colon\big[0,\frac{\epsilon_l}{\xi_l}\big]\to \R,t\mapsto \ln\left(\frac{1+t}{1-t}\right)$ is continuously differentiable on $\big[0,\frac{\epsilon_l}{\xi_l}\big]$ and $2 \leq h'(t)=2(1-t^2)^{-1}\leq 2\left(1-\left(\epsilon_l/\xi_l\right)^2\right)^{-1}$ for all $t\in \left[0,\frac{\epsilon_l}{\xi_l}\right].$
It follows that $h$ is a Lipschitz function and thus there exists $K_1 >0$ such that
$|h(t)-h(s)|\leq K_1|t-s|$ for every $s,t \in \big[0,\frac{\epsilon_l}{\xi_l}\big]$. Thus
\begin{equation*}
\tilde{\mu}_l(\vc v_l,\vc x_l) \leq h\big(t(\vc x_l)\big)=\big|h\big(t(\vc x_l)\big)-h\big(0\big)\big| \leq K_1\big|t(\vc x_l)-t(\vc v_l)\big| =\frac{K_1}{\xi_l}\norm{\vc x_l-\vc v_l}_{p_l} \quad \forall \vc x_l \in \overline{B_{\epsilon_l}^l(\vc v_l)}.
\end{equation*}
 In Equation $(1.20)$ in \cite{Nussb} it is shown
\begin{equation}\label{important_eq}
\norm{\vc x_l-\vc y_l}_{p_l}\leq 3\Big(e^{\tilde{\mu}_l(\vc x_l, \vc y_l)}-1\Big), \qquad \forall \vc x_l,\vc y_l\in\NS_{++}^{d_l}.
\end{equation}
In particular, as shown above, for every $\vc x_l\in\overline{B_{\epsilon_l}^l(\vc v_l)}$ we have $\tilde{\mu}(\vc v_l,\vc x_l)\leq \frac{\epsilon_lK_1}{\xi_l}$ and the function $t\mapsto e^{t}$ is Lipschitz on $\big[0,\frac{\epsilon_lK_1}{\xi_l}\big]$ (as it smooth and the derivative is bounded on the interval). So, there exists $K_2>0$ such that $|e^s-e^t|\leq \frac{K_2}{3}|s-t|$ for every $s,t\in\big[0,\frac{\epsilon_lK_1}{\xi_l}\big]$. It follows that 
\begin{equation*}
\norm{\vc x_l-\vc v_l}_{p_l}\leq 3\big|e^{\tilde{\mu}_l(\vc v_l, \vc x_l)}-3e^0\big|\leq K_2\big|\tilde{\mu}_l(\vc v_l,\vc x_l)-0| = K_2\tilde{\mu}_l(\vc v_l,\vc x_l) \qquad \forall \vc x_l\in\NS_{++}^{d_l}\cap\overline{B_{\epsilon_l}^l(\vc v_l)}.
\end{equation*}
Now, with $\tilde K \coloneqq \max\{K_2,\frac{(p_l-1)K_1}{\xi_l}\}$, since $p_l \geq 2$ for $l\in\msi$ by Remark \ref{pi_prop}, we get
\begin{equation*}
\norm{\vc x_l-\vc v_l}_{p_l}\leq\tilde K(p_l-1)\tilde{\mu}_l(\vc v_l,\vc x_l)=\tilde K\mu_l(\vc v_l,\vc x_l)\leq\tilde K^2\norm{\vc x_l-\vc v_l}_{p_l} \qquad \forall \vc x_l\in\NS_{++}^{d_l}\cap\overline{B_{\epsilon_l}^l(\vc v_l)}.
\end{equation*}
So, for every $l\in\msi$ there exists $r_l> 0$ and $C_l>0$ such that $
\norm{\vc u_l-\vc x_l}_{p_l}\leq C_l\mu_l(\vc x_l,\vc u_l) \leq C_l^2\norm{\vc u_l-\vc x_l}_{p_l}$ for all $\vc x_l \in B^l_{r_l}(\vc u_l)\cap\NS_{++}^{d_l}$. Now, note that $\norm{\vc v}_{\bar{p}} \coloneqq \sum_{l\in\msi}\norm{\vc v_l}_{p_l}$ is a norm on $\sR^{d-d_i}$. Since all norms are equivalent on finite dimensional spaces, there exists a constant $C_0$ such that $\norm{\vc v}_{\bar{p}} \leq C_0 \norm{\vc v}_G \leq C_0^2 \norm{\vc v}_{\bar p}$ for every $\vc v \in \sR^{d-d_i}$. Let $r = \min_{l\in\msi}\frac{r_l}{C_0}$ and $C = C_0\max_{l\in\msi}C_l$, then for all $\vc x \in\{\vc z\in \sR^{d-d_i}\mid \norm{\vc z -\vc u}_G<r\}\cap \NS^{d-d_i}_{++}$ we have $\norm{\vc x_l-\vc u_l}_{p_l} \leq \norm{\vc x-\vc u}_{\bar{p}} \leq C_0 \norm{\vc x-\vc u}_G < r_l$ for every $l\in\msi$, and thus 
\begin{align*}
\norm{\vc u-\vc x}_G &\leq C_0\norm{\vc u-\vc x}_{\bar p} \leq C_0\sum_{l\in\msi} C_l\mu_l(\vc u_l,\vc x_l) \leq C\mu(\vc u,\vc x) = C\sum_{l\in\msi} \mu_l(\vc u_l,\vc x_l) \\ &\leq C\sum_{l\in\msi} C_l\norm{\vc u_l-\vc x_l}_{p_l} \leq C\norm{\vc u-\vc x}_{\bar p}\max_{l\in\msi}C_l\leq C^2 \norm{\vc u-\vc x}_G .
\end{align*}
Divide the inequality by $C$ to conclude the proof.
\end{proof}
Now, we prove that $G$ is non-expansive with respect to the metric $\mu$ defined in Proposition \ref{metric_space}.
\begin{prop}\label{nonexpansive}
Let $f\geq 0$ a weakly irreducible tensor, $1 < p_1,\ldots,p_m < \infty$ such that there is $i \in [m]$ with $(m-1)p_i' \leq p_k$ for every $k\in\msi$ and $\mu$ defined as in Proposition \ref{metric_space}, then for every $\vc x,\vc y \in \NS_{++}^{d-d_i}$ we have $\mu\big(G(\vc x),G(\vc y)\big)\leq \mu(\vc x,\vc y)$.
\end{prop}
\begin{proof}
For $l\in\msi$, let $\phi_l \coloneqq \min_{j_l\in [d_l]}\frac{x_{l,j_l}}{y_{l,j_l}}$ and $\varphi_l \coloneqq \max_{j_l\in [d_l]}\frac{x_{l,j_l}}{y_{l,j_l}}$. Note that $\norm{\vc x_l}_{p_l}=\norm{\vc y_l}_{p_l}=1$ implies $\phi_l \leq 1 \leq \varphi_l$. By Lemmas \ref{order_preserving} and \ref{homos}, for every $l\in \msi$ and $j_l\in[d_l]$ we have 
\begin{equation*}
\phi_l^{-1}\left(\prod_{k\in\msi}\phi_k\right)^{p_i'}\big(s_{i,l,j_l}(\vc y)\big)^{p_l-1}\leq \big(s_{i,l,j_l}(\vc x)\big)^{p_l-1} \leq \varphi_l^{-1}\left(\prod_{k\in\msi}\varphi_k\right)^{p_i'}\big(s_{i,l,j_l}(\vc y)\big)^{p_l-1}.
\end{equation*}
It follows that for every $j_1\in [d_1],\ldots,j_m\in[d_m]$ we have 
\begin{equation*}
\prod_{l\in\msi}\phi_l^{p_l-1}\leq \prod_{l\in\msi}\left(\frac{s_{i,l,j_l}(\vc x)}{s_{i,l,j_l}(\vc y)}\right)^{p_l-1} \leq \prod_{l\in\msi}\varphi_l^{p_l-1},
\end{equation*}
where we have used $\phi_l \leq 1 \leq \varphi_l$ and $1 \leq p_i'(m-1)\leq p_l$ for every $l\in\msi$. Thus,
\begin{equation*}
e^{\mu(G(\vc x),G(\vc y))}=\frac{\displaystyle \prod_{l\in\msi}\max_{j_l\in [d_l]}\left(\frac{s_{i,l,j_l}(\vc x)}{s_{i,l,j_l}(\vc y)}\right)^{p_l-1}}{\displaystyle \prod_{l\in\msi}\min_{j_l\in [d_l]}\left(\frac{s_{i,l,j_l}(\vc x)}{s_{i,l,j_l}(\vc y)}\right)^{p_l-1}} \leq \frac{\displaystyle \prod_{l\in\msi}\varphi_l^{p_l-1}}{\displaystyle \prod_{l\in\msi}\phi_l^{p_l-1}}=e^{\mu(\vc x, \vc y)},
\end{equation*}
the desired inequality follows from the fact that $t\mapsto\ln(t)$ is an increasing function.
\end{proof}
Now, let us prove the convergence of the sequences produced by \algo{}. Note that Theorem \ref{convrate_thm} is a direct consequence of the next Theorem.
\begin{thm}\label{finally}
Let $f\geq 0$ be weakly irreducible and $1<p_1,\ldots,p_m<\infty$ such that there is $i\in [m]$ with $(m-1)p_i'\leq p_k$ for all $k\in [m]\setminus\{i\}$ and, if $m=2$, choose $i$ such that $p_i \leq p_k$ for $k\in [2]\setminus\{i\}$. Let $\vc x^*\in \NS^{d-d_i}$ be the unique strictly positive critical point of $Q_i$ in $\NS^{d-d_i}_{++}$. Let $(\lambda_-^k)_{k\in\N},(\lambda_+^k)_{k\in\N}\subset\R$ and $(\vc x^k)_{k\in\N}\subset\NS^{d-d_i}$ be the sequences produced by \algo{}. Then $(\vc x^k)_{k\in\N}$ converges to $\vc x^*$, $(\lambda_-^k)_{k\in\N},(\lambda_+^k)_{k\in\N}$ and $\big(Q_i(\vc x^k)\big)_{k\in\N}$ converge to $\norm{f}_{p_1,\ldots,p_m}$ and there is a norm $\norm{\cdot}_G$ on $\sR^{d-d_i},0<\nu <1$ and $k_0\in\N$ such that $\norm{\vc x^{k+1}-\vc x^*}_G\leq\nu\norm{\vc x^{k}-\vc x^*}_G$ for all $k\geq k_0$.
\end{thm}
\begin{proof}
First, we prove that $(\vc x^k)_{k\in\N}$ converges to $\vc x^*$. Let $\norm{\cdot}_G$ be the norm defined in Corollary \ref{Neighborhood} and $\mu_S:\sR^{d-d_i}\times \sR^{d-d_i}\to \R$ the metric defined by $\mu_S(\vc x,\vc y)\coloneqq \norm{\vc x-\vc y}_G$. Note that $\NS_{++}^{d-d_i}\subset \sR^{d-d_i}$ is connected and if $\mu$ is defined as in Proposition \ref{metric_space}, then $\mu$ is a metric on $\NS^{d-d_i}_{++}$. Moreover, Proposition \ref{loc_equiv} implies that the topology induced by $\mu$ on $\NS^{d-d_i}_{++}$ is the same as that inherited from $\big(\sR^{d-d_i},\mu_S\big)$. By Corollary \ref{Neighborhood} we know the existence of some $r_0>0$ and $0<\nu<1$ such that for every $\vc y^0 \in \sR^{d-d_i}$ with $\norm{\vc y^0-\vc x^*}_G<r_0$ holds
\begin{equation*}\lim_{k\to \infty} \mu_S(\vc y^k,\vc x^*)=\lim_{k\to \infty} \norm{\vc y^k-\vc x^*}_G\leq \lim_{k\to \infty} \nu^k\norm{\vc y^0-\vc x^*}_G =0 \quad \text{where } \vc y^{k+1} \coloneqq G(\vc y^k) \ \forall k\in\N.
\end{equation*}
Now, let $\norm{\vc v}_{\bar{p}} \coloneqq \sum_{l\in\msi}\norm{\vc v_l}_{p_l}$, $\mu_l$ defined as in Equation \eqref{def_mu_l}  and $C>0$ such that $\norm{\vc v}_G \leq C\norm{\vc v}_{\bar{p}}$ for every $\vc v \in \sR^{d-d_i}$. From Equation \eqref{important_eq} we know that $\norm{\vc x_l-\vc z_l}_{p_l} \leq 3\Big(e^{\frac{\mu_l(\vc x_l,\vc y_l)}{p_l-1}}-1\Big)$ for every $\vc x_l,\vc y_l \in \R^{d_l}$ with $\norm{\vc x_l}_{p_l}=\norm{\vc y_l}_{p_l}=1$. By Remark \ref{pi_prop}, we know that $p_k\geq 2$ for all $k\in\msi$. It follows that for every $\vc x,\vc y \in \NS_{++}^{d-d_i} \cap \big\{\vc z \in \sR^{d-d_i}\mid \norm{\vc z-\vc x^*}_G<r_0\big\}$ we have
\begin{align*}
\norm{\vc x-\vc y}_G &\leq C \norm{\vc x-\vc y}_{\bar{p}} \leq C \sum_{l\in\msi}3\left(e^{\frac{\mu_l(\vc x_l,\vc y_l)}{p_l-1}}-1\right) \leq C\sum_{l\in\msi}3\left(e^{\mu_l(\vc x_l,\vc y_l)}-1\right) \\
&\leq 3C(m-1)\left(e^{\sum_{l\in\msi}\mu_l(\vc x_l,\vc y_l)}-1\right) =  3C(m-1)\big(e^{\mu(\vc x,\vc y)}-1\big).
\end{align*}
In particular, the function $\varphi:\{t\in\R\mid t \geq 0\} \to \{t\in\R\mid t \geq 0\}$ defined by $\phi(t)=3C(m-1)\big(e^t-1\big)$ is continuous and satisfies $\varphi(0)=0$. Finally, note that by Proposition \ref{nonexpansive} we know that $\mu\big(G(\vc x),G(\vc y)\big)\leq \mu(\vc x,\vc y)$ for every $\vc x,\vc y\in\NS_{++}^{d-d_i}$. So, we may apply Lemma \ref{golden_lemma} and ensure that $\lim_{k\to \infty}\mu(\vc x^{*},\vc x^k)=0$ for every choice of $\vc x^0\in\NS^{d-d_i}_{++}$, i.e. $(\vc x^k)_{k\in\N}$ converges to $\vc x^*\in\NS^{d-d_i}_{++}$. From Theorem \ref{gen_sum_PF}, we know that $\norm{f}_{p_1,\ldots,p_m}=Q_i(\vc x^*)$ and since $\mbd_i,\Mbd_i$ and $Q_i$ are continuous functions on $\NS^{d-d_i}_{++}$, we have 
\begin{equation*}
\lim_{k\to \infty}\lambda_+^k=\lim_{k\to \infty}\big(\Mbd_i(\vc x^k)\big)^{\frac{1}{p_i'(m-1)}} =\big(\Mbd_i(\vc x^*)\big)^{\frac{1}{p_i'(m-1)}}=\norm{f}_{p_1,\ldots,p_m}=\lim_{k\to \infty}\big(\mbd_i(\vc x^k)\big)^{\frac{1}{p_i'(m-1)}} = \lim_{k\to \infty}\lambda_-^k,
\end{equation*}
and $\lim_{k\to \infty}Q_i(\vc x^k)=Q_i(\vc x^*)=\norm{f}_{p_1,\ldots,p_m}$. Finally, since $(\vc x^k)_{k\in\N}$ converges to $\vc x^*$, there exists $k_0>0$ such that $\norm{\vc x^k-\vc x^*}_G<r_0$ for every $k\geq k_0$ and thus $\norm{\vc x^{k+1}-\vc x^*}_G\leq\nu \norm{\vc x^k-\vc x^*}_G$ for every $k\geq k_0$.
\end{proof}
\section{Experiments}\label{numexp}
We compare the $\algo$ with the Power Method (PM) proposed by Friedland et al. in \cite{Fried}. The PM computes the singular values of nonnegative weakly irreducible tensors for the special case $p_1= \ldots = p_m$. The sequence $(\vc v^k)_{k \in \N} \subset \sR^d$ produced by this algorithm can be formulated as $\vc w^{k+1} = \big(\sigma_1(\vc v^k),\ldots,\sigma_m(\vc v^k)\big), \vc v^{k+1} = \frac{\vc w^k}{\vc n^T\vc w^k}$ for $k \geq 1$ and some vector $\vc n \in \sR^{d}$ with $\vc n >0$. For all our experiments we took $\vc n = (1,1,\ldots,1)$ and $\vc v^0 =\frac{(1,1,\ldots,1)}{d_1+\ldots+d_m}\in\sR^d$ and $\vc x^0 = \left( \frac{(1,1,\ldots,1)}{\norm{(1,1,\ldots,1)}_{p_1}},\ldots,\frac{(1,1,\ldots,1)}{\norm{(1,1,\ldots,1)}_{p_m}}\right)\in \sR^{d-d_i}$ as starting points. Plots show linear convergence as stated in Theorem \ref{finally}. Numerical experiments on randomly generated tensors showed that $\algo$ converges usually quicker than PM. Note also that the computation of one iteration of $\algo$ requires to go two times over all entries of $f$ while PM requires it only one time. In Figure \ref{fig_norm}, we show the convergence rate of both algorithms for $p = p_1 = \ldots = p_m \in \{3,4,5\}$ and the weakly irreducible tensor $f\in\R^{2\times 3 \times 4}$ defined by 
\begin{equation*}
f_{1,2,1}=806, f_{1,3,1} =761,f_{1,3,4}=3, f_{2,1,1}=833, f_{2,2,2}=285,f_{2,3,3}=176 \quad \text{ and }\quad  f_{j_1,j_2,j_3}=0\  \text{ else}.
\end{equation*}
\begin{figure}
\begin{center}
\begin{minipage}[t]{0.49\linewidth}
        \includegraphics[height = 6.2cm]{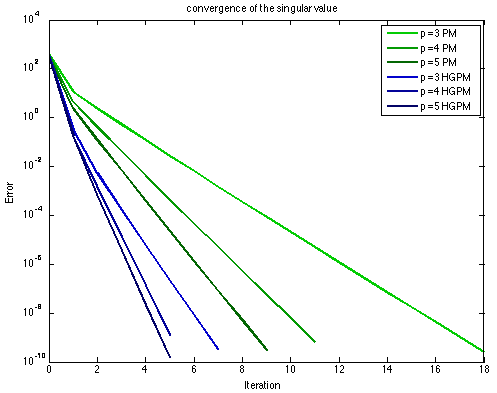}
\caption{Plot of the error $\big|Q(\vc v_1^k,\ldots,\vc v_{m}^k\big)-Q(\vc x^*)\big|$ (in green) and the error $\big|Q\big(\vc x_1^k,\ldots,\vc x_{i-1}^k,\sigma_i(\vc x^k),\vc x_{i+1}^k,\ldots,\vc x_m^k\big)-Q(\vc x^*)\big|$ (in blue) versus the number of iterations $k$ on a semilogarithmic scale.}
\label{fig_stopcrit}
    \end{minipage}\hfill
    \begin{minipage}[t]{0.49\linewidth}
\includegraphics[height = 6.2cm]{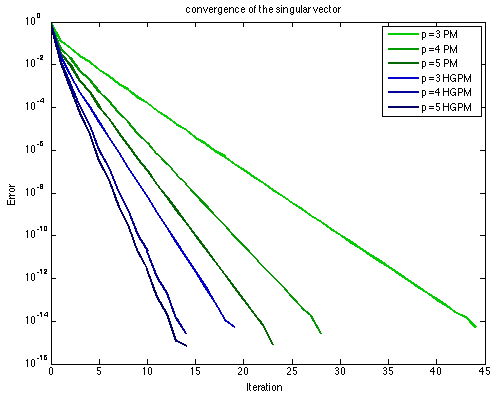}
\caption{Plot of the error $\norm{\vc v^k-\vc x^*}_2$ (in green) and the error $\Norm{\big(\vc x_1^k,\ldots,\vc x_{i-1}^k,\sigma_i(\vc x^k),\vc x_{i+1}^k,\ldots,\vc x_m^k\big)-\vc x^*}_2$ (in blue) versus the number of iterations $k$ on a semilogarithmic scale..}
\label{fig_norm}
    \end{minipage}
\end{center}
\end{figure}
\section*{Acknowledgements}
A. G. and M. H. acknowledge support by the ERC project NOLEPRO.
\section*{References}
\bibliography{tpn_bib.bib}
 \end{document}